\input ./style/arxiv-general.cfg
\documentclass[aop,MSNbibl,seceqn,dvips]{arximspdf}
\makeatletter
   \@ifpackageloaded{graphicx}{}{\usepackage{graphicx}}
\makeatother

\doi{10.1214/13-AOP903} 
\volume{43}
\issue{3}
\pubyear{2015}
\firstpage{1456}
\lastpage{1492}

\makeatletter
\renewcommand{\mid}{|}
\newcommand{\rrvert}{\vert}
\newcommand{\llvert}{\vert}
\def\E{{\mathbb{E}}}
\def\N{{\mathbb{N}}}
\def\P{{\mathbb{P}}}
\def\R{{\mathbb{R}}}
\def\Z{{\mathbb{Z}}}
\newcommand{\F}{{\mathfrak{F}}}
\newtheorem{teo}{Theorem}[section]
\newtheorem{cor}[teo]{Corollary}
\newtheorem{lem}[teo]{Lemma}
\newtheorem{lemma}[teo]{Lemma}
\newtheorem{prop}[teo]{Proposition}
\makeatother

\begin{document}
\begin{frontmatter}

\title{On unbounded invariant measures of stochastic~dynamical systems\thanksref{T1}}
\runtitle{On unbounded invariant measures of SDS}

\begin{aug}
\author[A]{\fnms{Sara} \snm{Brofferio}\ead[label=e1]{sara.brofferio@math.u-psud.fr}\thanksref{T2}}
\and
\author[B]{\fnms{Dariusz} \snm{Buraczewski}\corref{}\ead[label=e2]{dbura@math.uni.wroc.pl}\thanksref{T3}}
\runauthor{S. Brofferio and D. Buraczewski}
\affiliation{Universit\'e Paris-Sud and Uniwersytet Wroclawski}
\address[A]{Laboratoire de Math\'ematiques et IUT de Sceaux\\
Universit\'e Paris-Sud\\
91405 Orsay Cedex\\
France\\
\printead{e1}} 
\address[B]{Instytut Matematyczny\\
Uniwersytet Wroclawski\\
pl. Grunwaldzki 2/4\\
50-384 Wroclaw\\
Poland\\
\printead{e2}}
\end{aug}
\thankstext{T1}{Supported in part by Polonium project 7946/2010.}
\thankstext{T2}{Supported in part by ERC starting Grant GA 257110 ``RaWG.''}
\thankstext{T3}{Supported in part by NCN Grant DEC-2012/05/B/ST1/00692.}

\received{\smonth{4} \syear{2013}}
\revised{\smonth{12} \syear{2013}}

%
\begin{abstract}
We consider stochastic dynamical systems on $\R$, that is, random
processes defined by $X_n^x = \Psi_n(X_{n-1}^x)$, $X_0^x=x$, where
$\Psi
_n$ are i.i.d. random continuous transformations of some unbounded
closed subset of $\R$. We \mbox{assume} here that $\Psi_n$ behaves
asymptotically like $A_n x$, for some random positive number $A_n$ [the
main example is the affine stochastic recursion $\Psi_n(x) = A_n
x+B_n$]. Our aim is to describe invariant Radon measures of the process
$X_n^x$ in the critical case, when $\E\log A_1=0$. We prove that those
measures behave at infinity like $\frac{dx}{x}$. We study also the
problem of uniqueness of the invariant measure. We improve previous
results known for the affine recursions and generalize them to a larger
class of stochastic dynamical systems which include, for instance,
reflected random walks, stochastic dynamical systems on the unit
interval $[0,1]$, additive Markov processes and a variant of the
Galton--Watson process.
\end{abstract}

%

\end{frontmatter}

\section{Introduction}\label{sec1}

\subsection{Stochastic dynamical systems}\label{sec1.1}
Let $\F$ be the semigroup of continuous transformations of an unbounded
closed subset $\mathcal{R}$ of the real line $\R$
endowed with the topology of uniform convergence on compact sets. In
the most interesting examples, $\mathcal{R}$ is the real line, the
half-line $[0,+\infty)$ or the set of natural numbers $\N$. Given a
regular probability
measure $\mu$ on $\F$, we define the stochastic dynamical system (SDS)
on $\mathcal{R}$ by
%
%
\begin{eqnarray}
\label{recursion} X_0^x&=& x;
\nonumber
\\[-8pt]
\\[-8pt]
X_n^x&=& \Psi_n\bigl(X_{n-1}^x
\bigr),\nonumber
\end{eqnarray}
where $\{\Psi_n \}$ is a sequence of i.i.d. random functions,
distributed according to $\mu$.

The aim of this paper is to study conditions for the existence and
uniqueness, as well as behavior at infinity, of an invariant infinite
Radon measure of the process~$X_n^x$, that is, of a measure $\nu$
on $\R$ such that
%
%
\begin{equation}
\label{eqinvariance} \mu*_{\mathfrak F}\nu(f) = \nu(f)
\end{equation}
for any $f\in C_C(\R)$, where
\[
\mu*_{\mathfrak F}\nu(f) = \int_\R\E\bigl[ f
\bigl(X_1^x\bigr) \bigr] \nu(dx) = \int
_\F\int_{\R} f\bigl(\Psi(x)\bigr)
\nu(dx)\mu(d\Psi). %
\]

There is quite an extensive literature on the case when the process
$X_n$ is positive recurrent, that is, it possesses an invariant \textit
{probability} measure. The existence of such a measure can be proved
supposing that the process has some \mbox{contractive} property (e.g.,
if $\Psi_n$ are Lipschitz mappings with Lipschitz coefficients $L_n =
L(\Psi_n)$ and $\E[\log L_1] <0$), \cite{DF}).
This invariant probability measure is well described in several
specific cases, such as affine recursions [i.e., $\Psi(x)=Ax+B$],
namely in the seminal paper of Kesten \cite{Kesten}. Goldie
\cite{Go} and recently Mirek \cite{M} generalized Kesten's theorem
to stochastic recursions such that $\Psi(x)$ behaves like $Ax$ for
large $x$. They proved that if $\E A^\kappa= 1$ (and some other
hypotheses are satisfied), then
\[
\lim_{z\to\infty} z^\kappa\nu\bigl\{x\dvtx |x|>z\bigr\} = C_+ > 0.
\]
In\vspace*{1.5pt} other words, the measure $\nu$ is close at infinity to $\frac
{C_+dx}{x^{1+\kappa}}$.

Less is known for the null recurrent case, especially in a general
setting. Existence and uniqueness of an invariant Radon measure have
been the topic of two recent works: Deroin et al. \cite{DKNP} on
symmetric SDS of homeomorphism of $\R$, and Peign\'e and Woess
\cite{PW1} on the phenomenon of local contraction. We refer to them for
a more complete bibliography on the subject. As in the contracting
case, affine recursions is one of the first models being systematically
approached. A seminal paper in this area is the one of Babillot,
Bougerol and Elie \cite{BBE}. They proved existence and uniqueness of a
Radon measure and gave a first result on its behavior at infinity.

The goal of the present work is twofold. First of all, we investigate
the behavior at infinity of invariant measures, and for a large class
of SDSs, we generalize and improve results known for affine recursions.
Second, we consider the problem of uniqueness of the invariant measure.
We give a relatively simple criterium that can be applied for very
concrete examples.

\subsection{Behavior at infinity}\label{sec1.2}

It turns out that to prove existence and to describe the tail of the
measure it is sufficient to control the maps that generate the SDS near
infinity. In particular, we suppose that they are asymptotically
linear, in the sense that there exists $0\leq\alpha<1$ such that for
all $\psi\in\mathfrak{F}$
{\renewcommand{\theequation}{AL$^\alpha$}
\begin{equation}\label{ALal}
\bigl|\psi(x)-A_{\alpha}(\psi) x\bigr|\leq
B_\alpha(\psi) \bigl(1+|x|^\alpha\bigr)\qquad\mbox{for all }x
\in\mathcal{R}
\end{equation}}\setcounter{equation}{2}%
%
with $A_\alpha(\psi)$ and $B_\alpha(\psi)$ strictly positive. We study
here the critical case, that is, $\E[ \log A_\alpha] = 0$.

Existence of an invariant measure supported in $\mathcal{R}$ is
relatively easy to deduce from the well-known literature, because in
this case the SDS is bounded by a recurrent process (we give more
details in Section~\ref{ssecexistence}). The main result of the paper
is the description of the tail of invariant measures at infinity.

%
\begin{teo}
\label{mthm} Suppose that there exists $0\leq\alpha<1$ such that the
maps $\Psi_n$ satisfy (\ref{ALal}) $\mu$-a.s. and that
%
%
\begin{eqnarray}
\label{H1} && \E[\log A_\alpha] = 0\quad\mbox{and}\quad \P[ A_\alpha=
1]<1,
\\
\label{H2} && \E\bigl[\bigl(|\log A_\alpha| + \log^+|B_\alpha
|\bigr)^{2+\varepsilon
} \bigr]<\infty,
\\
&& \mbox{the law of $\log A_\alpha$ is aperiodic, that is,}
\nonumber\\[-8pt]\label{H3} \\[-8pt]
&&\qquad \mbox{there is no $p\in\R$ such that $\log A_\alpha\in p\Z$ a.s.}\nonumber
\end{eqnarray}
Let $\nu$ be an invariant Radon measure $\nu$ for the process $\{
X_n^x\}
_n$. Suppose that $\nu$ is supported by $\mathcal{R}$ and it is
positive on any neighborhood of $+\infty$. Then the family of dilated
measures $\delta_{z^{-1}}*\nu(I):=\nu(zI)$
converges vaguely on $\R^*_+=(0,\infty)$ to $C_+\frac{da}a$ as $z$
goes to
infinity for some $C_+>0$, that is,
\[
\lim_{z\to\infty} \int_{\R_+^*} \phi
\bigl(z^{-1}u\bigr)\nu(du) = C_+ \int_{\R^*_+} \phi(a)
\frac{da}a
\]
for any $\phi\in C_C(\R^*_+)$.
\end{teo}

The key example of an asymptotic linear SDS is the affine recursion
(called also the random difference equation). Then $\mathfrak{F}$ is
the set
of affine mappings of the real line $\Psi(x)=Ax+B$ with $A> 0$ and the
process is given by the following formula:
%
%
\begin{equation}
\label{eqaffinerecursion} X_n^x = A_n
X_{n-1}^x + B_n, \qquad X_0^x
= x.
\end{equation}
Our results are also valid for Goldie's recursions, for example, $\Psi
(x) =\break  \max\{Ax,B\}+C$ (with $A>0$)
and $\Psi(x) = \sqrt{A^2x^2+Bx+C}$ (with $A,B,C$ positive). Since the
problem can be reduced, without any loss of generality, to the case
$\alpha=0$ (see Lemma~\ref{lemalphatozero}), our hypotheses
essentially coincide, in the one-dimensional situation, with the class
introduced by Mirek \cite{M}. Our main theorem should be viewed as
an analog of Kesten's and Goldie's results in the critical case.

Other interesting examples can be obtained conjugating asymptotic
linear systems with an appropriate homeomorphism. For instance, our
result can also be applied to describe invariant measures of SDS on the
interval generated by functions that have the same derivative at the
two extremities. Theorem~\ref{mthm} also says that invariant measures
of SDS on $[0,+\infty)$ generated by mappings exponentially asymptotic
to translations, that is,
\[
\bigl|\psi(x)-x+u_\psi\bigr|\leq v_\psi{\mathrm{e}}^{-x}
\qquad\forall x\geq0
\]
behave at infinity as the Lebesgue measure $dx$ of $\R$, if $\E
(u_\phi
)=0$. This result can be compared with the Choquet--Deny theorem saying
that the only invariant measure for centered random walks on $\R$ is
the Lebesgue measure. Another interesting process that is $\alpha
$-asymptotically linear for $\alpha>1/2$ is a Galton--Watson evolution
process with random reproduction laws.
In Section~\ref{secexamples}, we give more details on the different examples.

Let us mention that in our previous papers \cite{B,BBD,BBD2} we have
already studied the behavior at infinity of the
invariant measure $\nu$ for the random difference equation (\ref
{eqaffinerecursion}). However,
the main results were obtained there under much stronger assumptions,\vspace*{1pt}
namely we assumed existence of exponential
moments, that is, $\E[A^\delta+ A^{-\delta}+|B|^\delta]<\infty$
for some $\delta>0$.
Theorem~\ref{mthm} improves all our previous results for affine
recursions and describes the asymptotic behavior of $\nu$ under optimal
assumptions, that is, the weakest-known conditions implying existence
of the invariant measure \cite{BBE}. To our knowledge, for all the
other recursions even partial results are not known.

We would like also to remark that, in the contracting case, Kesten's
theorem requires moment of order at least $\kappa$ and, as far as we
know, there exist no results on the behavior of the tail of the
invariant probability when the measure is supposed to have only
logarithmic moment.

The proof of Theorem~\ref{mthm} is given in Sections~\ref{secbounds}
and~\ref{sechomo}. In order to describe $\nu$ at infinity, we give
first an upper bound of this measure and prove some regularity
properties of its quotient. The techniques we use in the present paper
are more powerful than those applied in \cite{BBD}, and are heavily
based on the renewal theory for random walks on the affine group. Among
other results, we prove directly that $\nu[-z,z]$ grows as $\log z$
(Proposition~\ref{propboundnu}).
Next, in Section~\ref{sechomo}, we consider the Poisson equation for
the additive convolution on $\R$
\[
f(x) = \bar\mu* f(x) + g(x), %
\]
where $f(x) = \int\phi(e^{-x}u)\nu(du)$ for some $\phi\in C_C(\R^*_+)$
and $\bar\mu$ is the law of $-\log A$. Notice that the asymptotic
behavior of $f$ and $\nu$ is the same, therefore, it is sufficient to
study $f$.
In the contrast to \cite{BBD},
we do not explicitly solve this equation. We apply techniques borrowed
from the work of Durrett and Liggett \cite{DL}
(see also Kolesko~\cite{K}), reduce the problem to the classical
renewal equation with drift and deduce its asymptotic behavior from the
renewal theorem.

\subsection{Uniqueness of the invariant measure}\label{sec1.3}
Another fundamental question is to determine whether the invariant
measure is unique or not. The nature of this problem is different from
the ones we have considered so far. In fact, uniqueness depends on the
local behavior of the system and it is no more sufficient to control
the random maps only at the infinity.

In the noncontracting case, this problem was studied first by
Babillot, Bougerol and Elie \cite{BBE} in the context of the affine
recursion and they proved uniqueness under the assumptions of Theorem
\ref{mthm}. Relying on their ideas Benda \cite{B} studied in full generality
recurrent and locally contractive SDSs. The SDS is called recurrent if
there exists a closed set $L$
such that every open set intersecting $L$ is visited by $X_n^x$
infinitely often with probability 1.
The SDS is locally contractive if for any $x,y\in\R$ and every compact
set $K\subset\R$,
%
%
\begin{equation}
\label{eqlocalcontraction} \lim_{n\to\infty} \bigl|X_n^x-X_n^y\bigr|
\cdot{\mathbf1}_K\bigl(X_n^x\bigr) = 0 \qquad
\mbox{almost surely}.
\end{equation}
Benda \cite{B} proved that if $\{X_n^x\}$ is a recurrent and locally
contractive SDS, then it possesses a unique (up to a multiplicative constant)
invariant Radon measure. He did not publish his results, however, they
have been recently incorporated, with a complete and simplified proof,
into two papers of Peign\'e and Woess \cite{PW1,PW2}, where they also
investigated ergodicity of SDS generated by Lipschitz maps with
centered Lipschitz's coefficient.

Our aim is to consider very concrete families of Lipschitz mappings of
$\mathbb{R}_+$, as the one presented in Goldie's work \cite{Go}. Although
recurrence of the corresponding SDSs is immediate, the main obstacle in
applying Benda's theorem is the local contraction hypothesis (\ref
{eqlocalcontraction}). In \cite{PW2}, the authors considered the reflected
affine stochastic recursion, being a mixture of the reflected random
walk (described below) and the affine stochastic recursion [defined in
(\ref{eqaffinerecursion})]. Unfortunately, the method of hyperbolic
extensions they introduce cannot be applied to dynamical systems, whose
dependence on the affine recursion cannot be expressed in such a direct way.

A different approach can be found in \cite{DKNP}, where the authors
proved a local contraction property for a symmetric SDS generated by
homeomorphisms of $\R$. Their proof is very elegant but is heavily
based on the additional assumption that the SDS is generated by \emph
{invertible} mappings distributed according to a \emph{symmetric}
measure. In particular, their results cannot be applied to
noninvertible SDS, as the one generated by $\psi(x)=\max\{Ax,B\}+C$,
one of the most interesting in applications.

Our contribution to the subject is to give sufficient conditions for
uniqueness that can be applied to some concrete mappings of $\R
_+=[0,\infty
)$, such as $\psi(x)=\max\{Ax,B\}+C$ and other Goldie's recursions.

%
\begin{teo}
\label{teouniq} Suppose that $\mathcal{R}=[0,\infty)$, $\alpha=0$ and
that the hypotheses of Theorem~\ref{mthm} are satisfied. Assume
moreover that:
\begin{longlist}[(3)]
\item[(1)] there exists $\beta>0$ such that $\P(\Psi[0,+\infty)\subseteq
[\beta,+\infty))>0$;

\item[(2)] $A(\Psi) x\leq\Psi(x)\leq A(\Psi) x+B(\Psi)$ for all $x\geq0$;

\item[(3)] the functions $\Psi$ are Lipschitz and their Lipschitz
coefficients are equal to~$A(\Psi)$.
\end{longlist}
Then the SDS defined on $[0,\infty)$ by (\ref{recursion}) is locally
contractive.
Therefore, there exists a unique invariant Radon measure of the
process $\{X_n^x\}$ on $[0,+\infty)$.
\end{teo}

The proof of this theorem is contained in Section~\ref{secunique}.

\subsection{Reflected random walk}\label{sec1.4}
\label{secreflected}
The reflected random walk is the SDS defined for $x\in\R_+ =
[0,\infty)$, by
%
%
\begin{eqnarray}
\label{eqreflectedrandomwalk} Y_0^x & =& x,
\nonumber
\\[-8pt]
\\[-8pt]
Y_n^x &=& \bigl| Y_{n-1}^x -
u_n \bigr|,\nonumber
\end{eqnarray}
where $u_n$
is a sequence of i.i.d. real valued random variables with a given law
$\mu$.

If $u_n \ge0 $ a.s., then it was proved by Feller \cite{F} that this
process possesses a unique invariant probability
measure $\nu$, that is, a measure satisfying
\[
\mu* \nu(f) = \int_{\R_+} \int_{\R_+}
f\bigl(|x-y|\bigr)\nu(dx)\mu(dy) = \int_{\R
_+} f(x)\nu(dx) = \nu(f).
\]
Moreover, the measure $\nu$ can be explicitly computed: $\nu(dx) =
(1-F(x))\,dx$, for~$F$
being the distribution function of $\mu$. The process has been also
studied in more general settings
when $u_n$ admits also negative values (see Peign\'e, Woess \cite{PW1}
for recent results and a
comprehensive bibliography).

Here, we are interested in the critical case when $\E u_n = 0$. Peign\'
e and Woess \cite{PW1}
proved that if $\E(u_1^+)^{3/2}<\infty$, for\vspace*{1pt} $u_1^+ = \max\{
u_1,0\}$,
then the process
$\{X_n\}$ is recurrent on $\R_+$. As a consequence of Benda's theorem,
the process possesses
a unique invariant Radon measure $\nu$ on $\R_+$ (local contractivity
is easy to prove).
The reflected random walk can be transformed in an asymptotically
linear system by conjugating with an invertible function $s$ of
$[0,+\infty)$ such that $s(x)={\mathrm{e}}^x$ for large $x$. Then
$\psi(x)=s(|s^{-1}(x)-u|)$ is asymptotically linear with $A(\psi
)=e^{-u}$. Hence, Theorem~\ref{mthm} can be used to justify that the
invariant measure of $Y_n^x$
behaves at infinity like the Lebesgue measure. Nevertheless, in this
case, one can prove the same result under weaker moment assumptions and
a much simpler proof. A~short argument based
only on the duality lemma and the renewal theorem gives the following.

%
\begin{teo}
\label{mthm2}
Assume $\E u_1 = 0$, $\E(u_1^+)^{3/2}<\infty$, $\E
(u_1^-)^2<\infty$ and
the law $\mu$ of $u_1$
is aperiodic, then for every $\phi\in C_C(\R_+)$
\[
\lim_{x\to\infty} \int_{\R_+} \phi(u-x)\nu(du) =
C_+\int_{\R_+} \phi(u)\,du
\]
for some positive constant $C_+$.
\end{teo}

The proof of this theorem will be given in Section~\ref{secreflectedproof}.

We are grateful to the referees for their careful reading
of the manuscript and many helpful suggestions for improvement in the
presentation.

\section{Notation and preliminary results}\label{sec2}

\subsection{Reduction to condition (\texorpdfstring{\protect\ref{AL}}{AL})}\label{sec2.1}
Observe first that, conjugating the SDS with an appropriate function,
we can suppose without loss of generality that the distance of the
random map to a linear function is smaller than some constant. In fact,
we have the following lemma whose proof is postponed to \hyperref[appendix]{Appendix}.

%
\begin{lemma}\label{lemalphatozero}
Let $0\leq\alpha<1$. Suppose that $\psi$ satisfies
{\renewcommand{\theequation}{AL$^\alpha$}
\begin{equation}\label{ALal2}
\bigl|\psi(x)-A_{\alpha}x\bigr|\leq
B_\alpha\bigl(1+|x|^\alpha\bigr).
\end{equation}}%
Then the conjugate function $\psi_r=r\circ\psi\circ r^{-1}$, where
$r(x)=\operatorname{sign}(x)|x|^{1-\alpha}$, satisfies (\textup{AL}$^0)$ with
$A_0=A_\alpha^{1-\alpha}$. The appropriate constant $B_0$ can be chosen
such that $\log^+ B_0\leq C_\alpha(|\log A_\alpha|+\log^+B_\alpha+
1)$, for the constant $C_\alpha$ depending only on $\alpha$.
\end{lemma}

If $\psi$ is distributed according to $\mu$, the law $\psi_r$ is given
by $\mu_r=\delta_r* \mu*\delta_{r^{-1}}$, and if $\nu$ is a $\mu
$-invariant
measure then $\nu_r=\delta_r*\nu$ is $\mu_r$-invariant. Thus, if Theorem
\ref{mthm} holds for $\nu_r$, then it holds for $\nu$. Indeed
\begin{eqnarray*}
\lim_{z\to\infty} \int_{\R_+^*} \phi
\bigl(z^{-1}u\bigr)\nu(du) &=& \lim_{z\to
\infty} \int
_{\R_+^*} \phi\bigl(z^{-1}r^{-1}(u)\bigr)
\nu_r(du)
\\
&=& \lim_{z\to\infty} \int_{\R_+^*}
\phi\bigl(z^{-1}u^{1/(1-\alpha)}\bigr)\nu_r(du)
\\
&=& \lim_{z\to\infty} \int_{\R_+^*} \phi\bigl(
\bigl(z^{-(1-\alpha
)}u\bigr)^{1/(1-\alpha
)}\bigr)\nu_r(du)
\\
& =& C_+
\int_{\R^*_+} \phi\bigl(a^{1/(1-\alpha)}\bigr)\frac
{da}a
\\
&=& C_+ (1-\alpha) \int_{\R^*_+} \phi(a)\frac{da}a.
\end{eqnarray*}
In order to simplify our notation, we will suppose from now on that
$\alpha=0$, that is, for all $\psi\in\mathfrak{F}$
{\renewcommand{\theequation}{AL}
\begin{equation}\label{AL}
A(\psi)x-B(\psi)<\psi(x)<A(\psi)x+B(\psi)\qquad\mbox
{for all }x\in\mathcal{R}.
\end{equation}}\setcounter{equation}{0}%
Since $\mathcal{R}$ is closed, we can extend the property (\ref{AL}) to all
$x\in\R$ for a suitable continuous extension of $\psi$ to $\R$.
With a
slight abuse of notation, we will denote with the same letter (e.g.,
$\psi$), the map from $\mathcal{R}$ to $\mathcal{R}$ and its continuous
extension that verifies (\ref{AL}) for all $x\in\R$. In the same way, $\nu$
will be seen both as a measure on $\mathcal{R}$ and as a measure on
$\R
$ whose support is contained in $\mathcal{R}$.

\subsection{Comparison of $X_n^x$ with the affine recursion}\label{sec2.2}
We assume that the maps $A=A(\psi)$ and $B=B(\psi)$ from $\mathfrak{F}$
to $\R_+^*=(0,\infty)$ are measurable and that $\mathfrak{F}$ is a monoid
closed by composition. Assumption (\ref{AL}) implies
\[
\mathop{\lim_{x\to+\infty}}_{ x\in\mathcal{R}} \psi(x)/x=\mathop{\lim
_{x\to-\infty}}_{x\in\mathcal{R}}
\psi(x)/x=A(\psi),
\]
therefore, the map $A$ is a homomorphism from $\mathfrak{F}$ to $\R
_+^*$, that is, $A(\psi_1\circ\psi_2)=A(\psi_1)A(\psi_2)$. The choice of
$B$ is not unique and it can be chosen as big as needed.

Let $\{\Psi_n\}_{n=1}^\infty$ be an i.i.d. sequence of random variables
with values in $\mathfrak{F}$ of law~$\mu$. We are interested in the
study of the iterated stochastic function system
\[
X_n^x=\Psi_n\bigl(X_{n-1}^x
\bigr) = \Psi_n\cdots \Psi_1(x)\quad\mbox{and}\quad
X_0^x=x.
\]
%
If hypothesis (\ref{AL}) is satisfied, the trajectories of the process
$X_n^x$ can be dominated from below and from above by the affine recursions
%
%
\begin{equation}
\label{affrec} Z_n^x = A_n
Z_{n-1}^x - B_n\quad\mbox{and}\quad
Y_n^x = A_n Y_{n-1}^x +
B_n,
\end{equation}
where, to simplify our notation, we note $A_n=A(\Psi_n)$ and
$B_n=B(\Psi
_n)$. We will also assume, according to hypotheses of Theorem~\ref{mthm}, a logarithmic moment of order $2+\varepsilon$ and that $\log
A_1$ is nontrivially centered. Without any loss of generality, we can
also choose $B(\psi)$, such that
%
%
\begin{eqnarray}
\label{H4} &B_n \geq1\qquad\mbox{a.s.},&
\\
\label{H5} & \P(A_nx+B_n=x)<1\qquad\mbox{for all }x.&
\end{eqnarray}
In such a way, the two-dimensional process $(Z_n^x,Y_n^x)$ satisfies
all the assumptions required by Babillot, Bougerol and Elie \cite{BBE}.
Thus, it is recurrent, locally contractive and possesses a unique
invariant measure.

It will be convenient to use in the proof the language of groups.
Namely, let $G=\operatorname{Aff}(\R)=\R\rtimes\R^*_+$ be the group of all
affine mappings of $\R$, that is, the set of pairs $(b,a)\in\R\times
\R
^*_+$ acting on $\R\dvtx (b,a)\dvtx x \mapsto ax+b$. Then the group product
is given by the formula
\[
(b,a) \cdot\bigl(b',a'\bigr) = \bigl(b+ab',aa'
\bigr), %
\]
the identity element is $(0,1)$
and the inverse element is given by
\[
(b,a)^{-1} = (-b/a,1/a).
\]
Let $\mu_G$ be the probability distribution of $(B_n,A_n)$ on the
group $G$.
Then the random elements $g_n=(B_n,A_n)$ are i.i.d. random variables in
$G$ with law $\mu_G$.
We define the left and the right random walk on $G$:
%
%
\begin{equation}
\label{leftright} L_n = g_n\cdot \cdots \cdot g_1,
\qquad R_n = g_1\cdot \cdots \cdot g_n.
\end{equation}
Then $Y_n^x = L_n (x) $.

A very important role in our proofs will be played by the random walk
on $\R$ generated by $-\log A_i$, that is,
%
%
\begin{equation}
\label{sn} S_n = -(\log A_1+\cdots+ \log
A_n)
\end{equation}
(we put the sign minus to follow notations of our previous works).
Since \mbox{$\E\log A= 0$}, the random walk $S_n$ is recurrent. Moreover,
since we assume aperiodicity, the support of $S_n$ is just $\R$. We
often use the downward and upward sequence of stopping times
%
%
\begin{equation}
\label{stopping} \qquad l_n:=\inf\{k>l_{n-1}\dvtx S_k<S_{l_{n-1}}
\},\qquad t_n:=\inf\{ k>t_{n-1}\dvtx S_k\geq
S_{t_{n-1}}\}
\end{equation}
and $l_0=t_0=0$. Observe that $t_1$ and $l_1$ are almost surely finite,
but have infinite mean. On the other hand, hypothesis $\E(|\log
A|^{2+\varepsilon})<\infty$ guarantees that $S_{t_1}$ and $S_{l_1}$
are integrable (see \cite{CL}).

In the sequel, we will use, depending on the situation, different
convolutions. We define a convolution of a function $f$ on $\R$ with a
measure $\eta$ on $\R$ as a measure on $\R$ given by
%
%
\begin{equation}
\label{conv1} f*\eta(K) = \int_\R{\mathbf1}_K
\bigl(f(u)\bigr)\eta(du)= \eta\bigl(f^{-1}(K)\bigr).
\end{equation}
Given $z\in\R^*_+$ and a measure $\eta$ on $\R$, we define
%
%
\begin{equation}
\label{conv2} \delta_{z}*_{\R^*_+}\eta(K) = \int
_\R{\mathbf1}_K(zu) \eta(du)= \eta
\bigl(z^{-1}K\bigr).
\end{equation}

\subsection{Existence of an invariant measure}\label{sec2.3}\label{ssecexistence}

We conclude this section observing that the existence of the invariant
measure on $\mathcal{R}\subseteq\R$ for a SDS satisfying the
hypotheses of Theorem~\ref{mthm} follows immediately from recurrence of
the process $\{X_n^x\}$ and Lin's theorem \cite{L}.

More precisely, consider the positive operator $Pf(x)=\int f(\Psi
(x))\mu
(d\Psi)$ on $C_b(\mathcal{R})$. Then, since $Z^x_n\le X_n^x\le Y_n^x$
and $(Z^x_n, Y_n^x)$ is recurrent, the process $\{X_n^x\}$ is
recurrent, that is, there exists a nonnegative function $u\in C_c(\R)$
such that $\sum_{n=0}^\infty P^nu(x)=\infty$ for all $x$. Therefore, by
\cite{L}, there exists a nonnull invariant Radon measure $\nu$ on
$\mathcal{R}$ of the process $\{X_n^x\}$.

Observe that the support of this measure can be bounded (e.g.,
if the functions $\Psi$ fix the point $0$, then the Dirac measure at
$0$ is an invariant measure).
In this paper, we are interested in measures having unbounded support.
A sufficient
(but not necessary) condition to ensure that the invariant measure is
not bounded is to
assume that the random functions $\Psi$ do not fix a compact subset $C$
of $\R$ [i.e., there is no compact $C$ such that $\P(\Psi(C)\subseteq C)=1$].

\section{First bounds of the tail of the invariant measure}\label{sec3}
\label{secbounds}

We start to study the behavior of $\nu$ at infinity. In particular, we
will prove in this section that $\nu(dx)$ does not grow faster than
$\frac{dx}{x}$, the Haar measure of $\R_+^*$. The behavior of $\nu$ at
$\infty$
is related to the behavior of the family of measures $\delta
_{z^{-1}}*\nu$.
In this section, we prove the following.

%
\begin{prop}\label{propboundnu} Under the hypotheses of Theorem
\ref
{mthm}, we have the following:
\begin{longlist}[(3)]
\item[(1)] There exists $C_0>0$ such that
\[
\nu[-z,z]< C_0 ( 1+\log z)\qquad\mbox{for all } z>1.
\]
Moreover, if the support of $\nu$ is not bounded on the right, that is,
$\nu(z,+\infty)>0$ for all $z\in\R$, then:

\item[(2)] There exist $M>1$ and $\delta>0$ such that $\nu[z,zM]>\delta$
for all
$z\geq1$.

\item[(3)] For all $u_2>u_1>0$, there exists $C= C(u_1,u_2,M)>0$ such that
%
%
\begin{equation}
\label{eqcompactuniflogbound} \frac{\nu[{\mathrm{e}}^{x+y}u_1,{\mathrm
{e}}^{x+y}u_2]}{\nu
[{\mathrm{e}}^x,{\mathrm{e}}^x M]}< C ( 1+y)\qquad\mbox{for all } x>0, y>0.
\end{equation}
In particular, the family of measures
$\frac{1}{\nu[z,zM]}\delta_{z^{-1}}*\nu$ on $(0,+\infty)$ is vaguely
compact when $z$ goes to $+\infty$.
\end{longlist}
\end{prop}

There are two key arguments in the proof of this proposition. One is
the following lemma that we will use several times in the sequel.

%
\begin{lemma}\label{lemfond-bound-nu}
Let $\nu$ be a positive $\mu$-invariant measure on $\R$. Then for any
pair of intervals $V,U\subset\R$,
\[
\nu(V)\geq\P(T_\mathfrak{W}<\infty)\cdot\nu(U),
\]
where
\[
\mathfrak{W}=\mathfrak{W}(V,U)= \bigl\{\psi\in\F\mid\psi(U)\subset V
\bigr\}
\]
and $T_\mathfrak{W}$ is the stopping time defined by $T_\mathfrak
{W}=\inf\{n\geq0\dvtx \Psi_1\cdots\Psi_n \in\mathfrak{W}\}$.
\end{lemma}

\begin{pf}
Observe that the backward process
\[
M_n=\Psi_1\cdots\Psi_n*\nu(V)\qquad
M_0=\nu(V)
\]
is a positive martingale with respect to the filtration generated by
the $\Psi_n$. In fact,
\[
\E(M_n|\mathcal{F}_{n-1})=\Psi_1\cdots
\Psi_{n-1}*\mu*\nu(V) =\Psi_1\cdots\Psi_{n-1}*
\nu(V).
\]
Since $ (\Psi_1\cdots\Psi_{T_\mathfrak{W}})^{-1}(V)\supseteq U$,
for any fixed $n\in\N$, by the optional stopping time theorem,
\[
\nu(V)=\E(M_{T_\mathfrak{W}\wedge n})\geq\E\bigl({\mathbf1}_{\{
T_\mathfrak
{W}\leq n\}}\Psi_1
\cdots\Psi_{T_\mathfrak{W}}*\nu(V)\bigr) \geq\P(T_\mathfrak{W} <n) \nu(U).
\]
We let $n$ go to infinity to conclude.
\end{pf}

The other crucial observation
is that the backward recursion $\Psi_1\cdots\Psi_n(x)$ is controlled
by the right random walk $R_n$ on the affine group generated by the
product of $g_i=(B_i,A_i)$ [see (\ref{leftright})]. More precisely,
given $g\in\operatorname{Aff}(\R)$, we denote by $a(g)$ and $b(g)$ its
projections on $\R^*_+$ and $\R$, respectively, then
\[
a(R_n)x-b(R_n)\leq\Psi_1\cdots
\Psi_n(x)\leq a(R_n)x+b(R_n).
\]
We use these bounds to estimate the stopping time that appears in Lemma
\ref{lemfond-bound-nu}. In particular, as an immediate consequence of
the lemma above, we obtain the following.

%
\begin{cor}
\label{corst}
Let
\[
W=W(m_1,m_2,k_1,k_2)= \bigl
\{(B,A)\in\operatorname{Aff}(\R) \mid A k_2+B\leq m_2;  A
k_1-B \geq m_1 \bigr\}
\]
(see Figure~\ref{fig1}) and $T_W=\inf\{n\geq0\dvtx R_n\in W\}$. Then we have
\[
\nu(m_1,m_2) \ge\P[T_W <\infty]
\nu(k_1,k_2). %
\]
\end{cor}

%

\begin{pf}
The corollary follows from Lemma~\ref{lemfond-bound-nu}, taking
$U=[k_1,k_2]$, $V=[m_1,m_2]$ and noticing that $T_W\geq T_\mathfrak{W}$.
\end{pf}

Since the potential theory of the affine group is well understood, we
have enough tools to estimate $\P(T_W<+\infty)$ in many situations. For
a continuous and compactly supported function $f$ on $\operatorname
{Aff}(\R
)$, we define the potential
\[
U*\delta_g(f):=\E\Biggl[\sum_{n=0}^\infty
f(L_ng) \Biggr] = \E\Biggl[\sum_{n=0}^\infty
f(R_ng) \Biggr].
\]
\begin{figure}[t]

\includegraphics{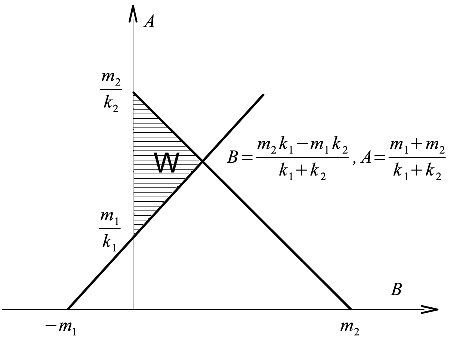}

\caption{The set $W=W(m_1,m_2,k_1,k_2)$.}\label{fig1}
\end{figure}

\noindent 
A renewal theorem for the potential $U$, that is, description of its
behavior at infinity, was given in \cite{BBE}, where the authors proved
that for all $h\in C_C(\operatorname{Aff}(\R))$:
%
%
\begin{equation}
\label{renewal} \lim_{a\to0} U*\delta_{(0,a)}(h) =
\nu_{G}\otimes\frac{dx}{x} (h)
\end{equation}
for $\nu_G$ being a suitable nontrivial multiple of the invariant
measure of the process $Y_n^x=L_n(x)$.

Now we are ready to prove the following lemma.

%
\begin{lemma}\label{lem-bound-stop-time} Suppose (\ref{H1}), (\ref
{H2}), (\ref{H4}) and (\ref{H5}). There exist a compact subset
\[
V_0=\bigl\{(B,A)\in\operatorname{Aff}(\R)\mid|B|<b_0,
a_0^{-1}<A<a_0 \bigr\}
\]
and a constant $\delta>0$ such that:
\begin{longlist}[(2)]
\item[(1)] if $W_z=(0,z)\cdot V_0= \{(B,A)\mid|B|<z b_0, z
a_0^{-1}<A<z a_0 \}$, then
\[
\P(T_{W_z}<\infty)> \delta
\]
for all $z\geq1$;
\item[(2)] if $V_z=V_0\cdot(0,z^{-1})= \{(B,A)\mid|B|< b_0, a_0^{-1}/z<A<
a_0/z \}$, then
\[
\P(T_{V_z}<\infty)> \frac{\delta}{1+\log z}
\]
for all $z\ge1$.
\end{longlist}
\end{lemma}

\begin{pf}
\textit{Step} 1. First observe that for every $V\subset\operatorname{Aff}(\R)$
%
%
\begin{equation}
\label{eqU-T} U\bigl(V^{-1}V\bigr) \P(T_V<\infty)\geq
U(V).
\end{equation}
In fact,
\begin{eqnarray*}
U(V)&=&\sum_{n=0}^\infty\P[R_n\in
V]=\E\Biggl[{\mathbf1}_{\{T_V<\infty\}
}\sum_{n=T_V}^\infty{\mathbf1}_{\{R_{T_V}R_n^{T_V}\in V\}} \Biggr]
\\
&\leq&\P(T_V<\infty) U
\bigl(V^{-1} V\bigr),
\end{eqnarray*}
where $R_n^l:=R_l^{-1}R_n=g_{l+1}\cdots g_n$.

\textit{Step} 2:
\textit{Proof of} (1).
By (\ref{eqU-T}), we write (assuming the
denominator is nonzero)
%
%
\begin{equation}
\P(T_{W_z}<\infty) \ge\frac{U(W_z)}{U(W_z^{-1}W_z)}=\frac
{U((0,z)\cdot
V_0)}{U(V_0^{-1}V_0)}.
\end{equation}
A simple calculation relates the right random walk on the affine group
to the reversed left random walk $\breve L_n= R_n^{-1} = g_n^{-1}\cdots
g_1^{-1}$.
Observe that for any $V\subset \operatorname{Aff}(\R)$ we have
\begin{eqnarray*}
U\bigl((0, z ) V\bigr) &=& \sum_n \P
\bigl[R_n\in(0, z ) V \bigr]=\sum_n\P
\bigl[R_n^{-1}\in V^{-1}\bigl(0, z^{-1}
\bigr) \bigr]
\\
&=& \sum_n \P\bigl[\breve L_n(0, z )
\in V^{-1} \bigr]=\breve U \bigl(V^{-1}\bigl(0,
z^{-1} \bigr) \bigr),
\end{eqnarray*}
where $\breve U$ is the potential of the reversed random walk $\breve
L_n$. Since the law of $g_n^{-1}$ is also centered and verifies the
hypotheses of \cite{BBE}, there exists a unique Radon measure $\breve
\nu_G$ on $\R$ invariant under $\breve\mu_G$, the law of $g^{-1} =
(B,A)^{-1}$.
Then by (\ref{renewal})
\[
\lim_{z\to+\infty}U\bigl((0, z ) V\bigr)= \lim_{z\to+\infty}
\breve U \bigl(V^{-1}\bigl(0, z^{-1} \bigr) \bigr) = \biggl(
\breve\nu_G\times\frac{dx}x \biggr) \bigl(V^{-1}
\bigr).
\]
We take sufficiently large $V_0$ such that
\[
U\bigl(W_z^{-1}W_z\bigr)=U
\bigl(V_0^{-1}V_0\bigr)>0\quad\mbox{and}\quad
\biggl( \breve\nu_G\times\frac{dx}x \bigl(V_0^{-1}
\bigr) \biggr)>0
\]
and, in view of (\ref{eqU-T}), we conclude.

\textit{Step} 3: \textit{Proof of} (2).
As in the previous step, by (\ref{eqU-T}),
we write
%
%
\begin{equation}
\label{eqmod34} \P(T_{V_z}<\infty) > \frac{U(V_z)}{U(V_z^{-1}V_z)}=
\frac{U
(V_0(0,z^{-1}) )}{U ((0,z)V_0^{-1}V_0(0,z^{-1}) )}.
\end{equation}
Now we have to estimate $U(V_z)$ from below 
and
$U(V_z^{-1}V_z)$ from above.
The latter is the most difficult part of the proof.

To deal with this second problem, we decompose the centered random walk
on the affine group in a contracting part and a dilating part using
ladder stopping times. This key idea has been applied in several
different ways in important works on the subject, for instance, \cite
{Grin,E,LPP,BBE}. We use here a potential theoretic version.
Let $\{\bar{g}_i\}$ be another sequence of i.i.d. elements of
$\operatorname{Aff}(\R)$ independent and of the same law as $\{g_i\}$.
We define
$\overline S_n, \bar t_k, \bar l_k$ as in (\ref{sn}) and (\ref{stopping}).
We claim that
%
%
\begin{equation}
\label{potential} U(f)=\E\Biggl[\sum_{n=0}^\infty
f(L_n) \Biggr]=\E\Biggl[\sum_{k,i=0}^\infty
f(\overline{R}_{\bar{l}_i}L_{t_k}) \Biggr].
\end{equation}
In fact, for $n> k$ define $L^k_n=g_n\cdots g_{k+1}$ and $L^k_k=e$.
Observe that
\[
\E\Biggl[\sum_{n=0}^\infty
f(L_n) \Biggr]=\E\Biggl[\sum_{k=0}^\infty
\sum_{i=t_k}^{t_{k+1}-1} f(L_i)
\Biggr]=\E\Biggl[\sum_{k=0}^\infty\E\Biggl[
\sum_{i=t_k}^{t_{k+1}-1} f\bigl(L^{t_k}_iL_{t_k}
\bigr) \big|L_{t_k} \Biggr] \Biggr].
\]
Since for fixed $k$, the sequence $\{L^{t_k}_{t_k+i}\}_{i\geq0}$ is
independent of $L_{t_k}$ and has the same law as $\{L_i\}_{i\geq0}$,
by the duality lemma (see Lemma 5.4 \cite{BBD}) we have
\[
\E\Biggl[\sum_{i=t_k}^{t_{k+1}-1} f
\bigl(L^{t_k}_iL_{t_k}\bigr)\big|L_{t_k}=g
\Biggr]=\E\Biggl[\sum_{i=0}^{t_{1}-1}
f(L_ig) \Biggr]=\E\Biggl[\sum_{i=0}^\infty
f(\overline{R}_{\bar{l}_i}g) \Biggr]
\]
and we obtain (\ref{potential}).

Observe that $\overline S_{\bar{l}_i}$ (resp., $S_{t_k}$) is a
random walk
with finite mean and negative (resp., positive) steps. Take $a,b>2$,
then by (\ref{potential}) and the classical renewal theorem \cite{F},
we have
\begin{eqnarray*}
&& U\bigl([-b,b]\times[1/a,a]\bigr)
\\
&&\qquad = \sum
_{k,i=0}^\infty\P\bigl[ b(\overline{R}_{\bar{l}_i}L_{t_k})
\le b; -\log a\le\overline{S}_{\bar{l}_i}+S_{t_k}\le\log a \bigr]
\\
&&\qquad =\sum_{k,i=0}^\infty\P\bigl[{
\mathrm{e}}^{-\overline S_{\bar{l}_i}}b(L_{t_k})+b(\overline{R}_{\bar{l}_i})
\leq b; -\log a\leq\overline{S}_{\bar{l}_i}+S_{t_k}\leq\log a
\bigr]
\\
&&\qquad \leq \sum_{k,i=0}^\infty\P\bigl[b(
\overline{R}_{\bar{l}_i})\leq b; -\log a\leq\overline{S}_{\bar{l}_i}+S_{t_k}
\leq\log a \bigr] \qquad\mbox{since $b(L_{t_k})\ge0$}
\\
&&\qquad = \sum_{i=0}^\infty\E\Biggl[{
\mathbf1}_{[b(\overline{R}_{\bar{l}_i})\leq
b] } \E\Biggl[\sum_{k=0}^\infty{
\mathbf1}_{\{-\log a\leq\overline
{S}_{\bar{l}_i}+S_{t_k}\leq\log a \}} |\bar{g}_i, i\geq0 \Biggr] \Biggr]
\\
&&\qquad \le C \log a \sum_{i=0}^\infty\P\bigl[b(
\overline{R}_{\bar{l}_i})\leq b \bigr].
\end{eqnarray*}
Since we assume $B\geq1$ a.s., we have for $i\geq1$:
\[
b(\overline{R}_{\bar{l}_i})=b\bigl(\overline{R}_{\bar{l}_{i-1}}
\overline{R}{}^{\bar{l}_{i-1}}_{\bar{l}_{i}}\bigr)={\mathrm
{e}}^{-\overline S_{\bar{l}_{i-1}}}b
\bigl(\overline{R}{}^{\bar{l}_{i-1}}_{\bar{l}_{i}}\bigr)+b(\overline
{R}_{\bar{l}_{i-1}})
\geq{\mathrm{e}}^{-\overline S_{\bar{l}_{i-1}}}.
\]
That is,
\begin{eqnarray*}
U\bigl([-b,b]\times[1/a,a]\bigr)&\leq& C \log a \Biggl(1+\sum
_{i=1}^\infty\P[ \overline S_{\bar{l}_{i-1}} \geq- \log b ] \Biggr)
\\
&\leq& C \log a (1+C\log b).
\end{eqnarray*}
Therefore, since
\[
V_z^{-1}V_z\subseteq\bigl\{ (B,A)\mid|B|
\leq2b_0a_0 z, a_0^{-2}\leq A\leq
a_0^2\bigr\},
\]
we obtain
\[
U\bigl(V_z^{-1}V_z\bigr)\leq K \log
a_0\bigl(1 + \log z+\log(2b_0a_0) \bigr).
\]
To estimate $U(V_z)$ from below as in the previous case, we just apply
the renewal theorem (\ref{renewal}). Plugging those estimates into
(\ref{eqmod34}), we conclude.
\end{pf}

\begin{pf*}{Proof of Proposition~\ref{propboundnu}}
\textit{Step} 1: \textit{Proof of} (1).
We apply Corollary~\ref{corst} with
$[k_1,k_2]=[-z,z]$ and $[m_1,m_2]=[-2b_0,2b_0]$ and consider,
according to the notation there, the subset of $\operatorname{Aff}(\R)$
\begin{eqnarray*}
W(-2b_0,2b_0,-z,z) &=& \bigl\{g\in\operatorname{Aff}(\R)\mid g
\bigl([-z,z]\bigr)\subseteq[-2b_0,2b_0] \bigr\}
\\
&=& \bigl
\{(B,A)\mid Az+B<2b_0 \bigr\}.
\end{eqnarray*}
This subset contains the set
\[
V_z= \biggl\{(B,A) \bigg| \frac{b_0^{-1}}{z}<A<\frac{b_0}{z}, |B| <
b_0 \biggr\}.
\]
We can apply Corollary~\ref{corst} and, choosing $b_0$ large enough,
Lemma~\ref{lem-bound-stop-time}(2) to conclude:
\[
\nu(-z,z)\leq\frac{\nu[-2b_0,2b_0]}{\P(T_{V_z}<\infty)}< C_0(1+\log z).
\]

\textit{Step} 2: \textit{Proof of} (2).
Take $M>1$ and $0<k_1<k_2$. Set $[m_1,m_2]=[ z, z M]$.
Then by Corollary~\ref{corst}
\[
\nu[ z, z M]\geq\P(T_{W_z}<\infty)\nu[k_1,k_2],
\]
where
\[
W_z=W( z, z M,k_1,k_2)=(0, z )W(1,
M,k_1,k_2)=:(0, z )W_1
\]
(see Figure \ref{fig2}). Observe that if $k_1$, $M$ and $M/k_2$ tend to infinity, then
\[
W_1=\bigl\{(B,A)\mid Ak_1- B>1, Ak_2+B<M
\bigr\}
\]
grows to $\operatorname{Aff}(\R)$. Thus,
there exists $C>0$ such that if $k_1\geq C$, $M>C$ and $M/k_2\geq C$,
the set $W_1$ contains the compact set $V_0$ defined
in Lemma~\ref{lem-bound-stop-time}. Therefore, $\P(T_{W_z}<\infty)$ is
uniformly bounded from below for large values of $z$.
Moreover, since we require the support of $\nu$ to be unbounded on the
right, one can choose $k_2$ such that $\nu[k_1,k_2] > 0 $ and we conclude.

\begin{figure}[t]

\includegraphics{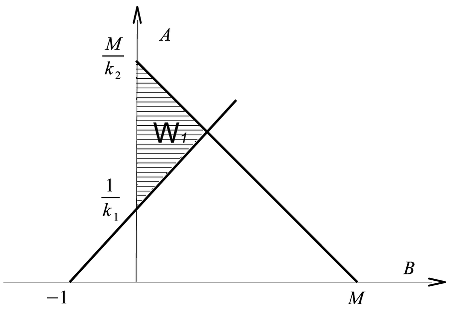}

\caption{The set $W=W(z,zM_1,k_1,k_2)$.}\label{fig2}
\end{figure}

\textit{Step} 3. \textit{Proof of} (3).
Let $a_0$, $b_0$ be sufficiently large
numbers such that Lemma~\ref{lem-bound-stop-time} holds. Take $M>\max
\{
2, 4 a_0^2\}$.

First suppose that $\frac{u_2}{u_1}<\frac{M}{ 4 a_0^2}$.
Take $[m_1,m_2]=[{\mathrm{e}}^x,{\mathrm{e}}^x M]$ and
$[k_1,k_2]= [{\mathrm{e}}^{x+y}u_1,{\mathrm{e}}^{x+y} u_2]$.
For $x>\log(b_0)$, the set
\begin{eqnarray*}
&& W\bigl({\mathrm{e}}^x,{\mathrm{e}}^x M,{
\mathrm{e}}^{x+y}u_1,{\mathrm{e}}^{x+y}
u_2\bigr)
\\
&&\qquad = \bigl\{(B,A)\in\operatorname{Aff}(\R) \mid A {
\mathrm{e}}^{x+y} u_2+B\leq{\mathrm{e}}^{x}M;
 A {\mathrm{e}}^{x+y} u_1-B \geq{\mathrm{e}}^{x}
\bigr\}
\end{eqnarray*}
contains the set
\[
V(y)= \biggl\{(B,A)\in\operatorname{Aff}(\R) \mid B< b_0,
\frac
{2}{{\mathrm{e}}^{y}
u_1} \leq A \leq\frac{M}{{\mathrm{e}}^{y} 2u_2} \biggr\}.
\]
Since $(\frac{M}{{\mathrm{e}}^{y} 2u_2})/(\frac{2}{{\mathrm{e}}^{y}
u_1})=\frac{M u_1}{
4u_2}>a_0^2$, we can apply Lemma~\ref{lem-bound-stop-time} and prove
that there exists $C>$ such that
\[
\frac{\nu[{\mathrm{e}}^{x+y}u_1,{\mathrm{e}}^{x+y}u_2]}{\nu
[{\mathrm{e}}^x,{\mathrm{e}}^x M]}\leq\frac{1}{\P
(T_{V(y)}<\infty)}< C ( 1+y)\qquad\mbox{for all } x>\log
b_0, y>0.
\]
By the previous steps, the last inequality is satisfied for $0 <x \leq
\log b_0$ and all $y>0$.

For general $U=[u_1,u_2]$ with $\frac{u_2}{u_1}\geq\frac{M}{ 4
a_0^2}$, we can deduce (\ref{eqcompactuniflogbound}) covering $U$
with a finite number of small intervals.
\end{pf*}

Since the law of $\log A$ is aperiodic, proceeding as in \cite{BBE} and
\cite{BBD}, one can prove that the family of quotient measures is
asymptotically invariant under the action of $\R^*_+$ and converges to
the Haar measure of $\R^*_+$.

%
\begin{cor}\label{corSV} Under the hypotheses of Theorem~\ref{mthm}
\[
\liminf_{z\to\infty} \delta_{z^{-1}}*\nu(\phi)>0,
\]
where $\phi$ is an arbitrary nonzero and nonnegative element of $C_
c(0,+\infty)$.

Furthermore, for $\phi_1,\phi_2 \in C_ c(0,+\infty)$ and $\phi_2$ not
identically zero,
%
%
\begin{equation}
\label{eqquot} \lim_{z\to\infty}\frac{\delta_{z^{-1}}*\nu(\phi
_1)}{\delta
_{z^{-1}}*\nu
(\phi
_2)}=
\frac{\int_{\R_+^*} \phi_1(a) ((da)/a)}{\int_{\R_+^*}
\phi
_2(a) ((da)/a)}.
\end{equation}
Therefore,
%
%
\begin{equation}
\label{eqSV} \lim_{x\to+\infty} \frac{\delta_{e^{-(x+y)}}*\nu(\phi
)}{\delta
_{e^{-x}}*\nu
(\phi)}=1
\end{equation}
and
\[
\frac{\delta_{e^{-(x+y)}}*\nu(\phi)}{\delta_{e^{-x}}*\nu(\phi
)}\leq K_\phi(1+y)\qquad\mbox{for all $x,y>0$}.
\]
In particular, the function $L(z)=\delta_{z^{-1}}*\nu(\phi)$ is
slowly varying.
\end{cor}

\begin{pf} For the reader's convenience, we present a sketchy proof
(see Proposition 2.2 \cite{BBD} for more details).
First, take a Lipschitz function $\Phi$ whose compact
support contains $(1,M)$ and let $L(z)=\delta_{z^{-1}}*\nu(\Phi)$.
Since the family of measures $\tilde\nu_z=\frac{1}{L(z)} \delta
_{z^{-1}}*\nu$
is vaguely compact, for every sequence we can extract its subsequence
$\tilde\nu_{z_n}$ convergent to a limit measure $\eta$.

For every Lipschitz compactly supported function $\phi$ and $\Psi\in
\mathfrak{F,}$ there exists a compact set $U=U(\phi,\Psi)$ such that
\[
\biggl| \phi\biggl(\frac{\Psi(u)}{z} \biggr)-\phi\biggl(\frac{ A
u}{z} \biggr)
\biggr| \le\frac{ B }z\cdot{\mathbf1}_U \biggl(\frac{ A u}z
\biggr)
\]
and
\begin{eqnarray*}
&& \lim_{n\to\infty} \frac{ | \int\phi(\Psi
(u)/{z_n} )\nu
(du)-\int\phi((A u)/{z_n} )\nu(du) | }{L( z_n)}
\\
&&\qquad \leq \lim
_{n\to\infty} \frac{C| z_n^{-1}b| \nu( a^{-1}z_n U)}{L(
z_n)}
\\
&&\qquad \leq C \eta\bigl(a^{-1} U\bigr)\cdot\lim_{n\to\infty} \bigl|
z_n^{-1}b\bigr|= 0.
\end{eqnarray*}
Thus, the function
\[
h(y)=\delta_{y}*\eta(\phi)=\lim_{n\to\infty}
\frac{\delta_{(0,
z_n^{-1}y)}*_G\nu
(\phi)}{L( z_n)}
\]
on $\R^*_+$ is superharmonic with respect to the action of $\mu_A$, the
law $A_1$.
Since $h$ is positive and continuous, by the Choquet--Deny theorem it
must be a constant function, that is $\delta_{a}*\eta(\phi)=\eta
(\phi)$
for every $a\in\R^*_+$. Because $\eta(\Phi)=1$, then $\eta$ is a fixed
multiple of the Haar measure of $\R_+^*$ and
\[
\lim_{z\to+\infty} \frac{\delta_{z^{-1}}*\nu(\phi)}{\delta
_{z^{-1}}*\nu
(\Phi)}=\frac{\int\phi(a)((da)/{a})}{\int\Phi(a)((da)/{a})}.
\]
This proves (\ref{eqquot}) and (\ref{eqSV}). In particular, if
$\phi
$ is nonzero, by Proposition~\ref{propboundnu}, we have
\[
\liminf_{z\to\infty} \delta_{z^{-1}}*\nu(\phi)\geq\int\phi
(a)\frac
{da}{a}\cdot\liminf_{z\to\infty} \delta_{z^{-1}}*
\nu(\Phi) >0.
\]

Take $k$ such that the support of $\phi$ is contained in $[1/k,k]$. Then
\[
\frac{e^{-(x+y)}*\nu(\phi)}{e^{-x}*\nu(\phi)}\leq\frac{\nu
[{\mathrm{e}}^x/M,{\mathrm{e}}^x
M]}{e^{-x}*\nu(\phi)} \frac{\nu[{\mathrm{e}}^{x+y}/k,{\mathrm
{e}}^{x+y} k]}{\nu[{\mathrm{e}}^x/M,{\mathrm{e}}^x
M]}\leq K (1+y),
\]
because the first quotient is bounded.
\end{pf}

\section{Homogeneity at infinity}\label{sec4}
\label{sechomo}

In this section, we finish the proof of Theorem~\ref{mthm}. The main
idea of the proof is similar to our previous papers \cite{BBD,Bu,BBD2}.
Given a nice function $\phi$ on $\R^*_+$ we define the function
\[
f(x) = \int_{\R^*_+} \phi\bigl(e^{-x}u\bigr)\nu(du).
\]
Behavior at infinity of the measure $\nu$ is coded in the asymptotic
behavior of $f$. To describe $f$, we consider it as a solution of the
Poisson equation
\[
\bar\mu*_\R f(x) = f(x) + g(x),
\]
where $\bar\mu$ is the law of $-\log A$ and the function $g$ is defined
by the equation above. We cannot use the classical renewal theorem,
since the measure $\bar\mu$ is centered. In our previous papers, we
expressed $f$ as a special potential of $g$. However, this approach was
technically involved and it was not possible to establish the optimal
hypotheses. Here, we apply ideas due to Durrett and Liggett \cite{DL},
who studying a similar equation and applying the duality lemma, were
able to reduce the problem to the classical renewal theorem. In
Proposition~\ref{lemlim-poiss-eq}, we determine weak assumptions in
the terms of the Poisson equation that enable to control the asymptotic
behavior of the solution.

In the second part of the section, we apply this result to our problem.
We show that there exist slight perturbations of the functions $f$ and
$g$ defined above which satisfy all the required conditions. Finally,
we deduce our main result proving that the tail of the measure $\nu$
converges at infinity.

%
\begin{prop}\label{lemlim-poiss-eq}
Let $\bar{\mu}$ be a centered probability measure on $\R$ with
finite moment of order $2+\epsilon$ for some $\epsilon>0$
and let $f$ be a continuous function on $\R$ such that
%
%
\begin{equation}
\label{eqcondf} 0\leq f(x) \leq C \bigl(1+x^+\bigr)\quad\mbox{and}\quad\int
_{-\infty}^y f(x)\,dx\leq C \bigl(1+y^+\bigr),
\end{equation}
where $x^+:=\max\{0,x\}$. Let $g$ be the continuous function on $\R$
defined by the Poisson equation:
%
%
\begin{equation}
\label{poissonequation} \bar\mu* f(x) = f(x) + g(x).
\end{equation}
Suppose also that $g$ is directly Riemann integrable,
then
%
%
\begin{equation}
\label{eqlim-poiss-eq-non-cent} \lim_{x\to+\infty} \E\bigl[ f(x+S_t)
\bigr]-f(x) = \frac{-1}{\E[S_{l}]}\int_\R g( x) \,dx,
\end{equation}
where $S_n$ is the random walk of law $\bar{\mu}$ and $t$ and $l$
are the stopping times
\[
t = \inf\{n> 0\dvtx S_n \ge0\}\quad\mbox{and}\quad l = \inf\{n > 0\dvtx
S_n < 0\}.
\]
Moreover, if $\int_\R g(x) \,dx=0$ and $\int_\R|x g(x)| \,dx <\infty$,
%
%
\begin{equation}
\label{eqlim-poiss-eq} \lim_{x\to+\infty} \E\biggl[ \int_x^{x+S_t}
f(z)\,dz \biggr] = \frac{1}{\E
[S_{l}]}\int_\R x g( x) \,dx.
\end{equation}
\end{prop}

The notion of directly Riemann integrable functions is fundamental in
renewal theory and allows to apply the classical renewal theorem to the
function $g$ (see, e.g., Feller \cite{F}).
The proof of this proposition will be given in \hyperref[appendix]{Appendix}.

Let $\nu$ be a $\mu$-invariant Radon measure on $\R$. We would like to
apply the previous proposition to the function $f(x)= \delta_{{\mathrm
{e}}^{-x}}*\nu
(\phi)$ for some fixed positive function $\phi\in C_C^1(\R^*_+)$.
Unfortunately, we are not able to justify that $f$ satisfies all the
required hypotheses. The main reason is that we are not able to control
local properties of a general measure $\nu$, namely its behavior near
0. Thus, the function $f$ may not be sufficiently integrable at
$-\infty$.
However, it turns out that one can slightly translate the measure $\nu
$ to overcome the problem.

For this purpose, given $\phi\in C_C^1(\R^*_+)$ and $w_0>0$ define
\begin{eqnarray*}
f_\phi(x) &:=& \int_{\R} \phi
\bigl(e^{-x}(u-w_0)\bigr) \nu(du),
\\
g_\phi(x) &:=& \bar\mu*_\R f_{\phi}(x) -
f_{\phi}(x).
\end{eqnarray*}
Observe that $f_\phi(x)=\delta_{{\mathrm{e}}^{-x}}*\nu_0(\phi)$ where
$\nu_0$ is the
measure $\nu$ translated by $w_0$:
\[
\nu_0(\phi)=\int_{\R} \phi(u-w_0)
\nu(du),
\]
that is, the invariant measure of the SDS obtained by conjugating the
original one with the translation by $w_0$:
\[
\psi_0(x)= \psi(x+w_0)-w_0.
\]
Denote by $\mu_0$ its law. Observe that
$A(\psi_0)=A(\psi)$ and we can choose $B(\psi_0)=Aw_0+w_0+B$,
hence $\mu_0$ satisfies our main hypotheses if $\mu$ does. Since the
translation does not change the asymptotic behavior,
the measures $\nu_0$ and $\nu$ behave in the same way at $+\infty$, namely
%
%
\begin{equation}
\label{eqnu-f} \lim_{x\to+\infty} f_\phi(x)-
\delta_{{\mathrm{e}}^{-x}}*\nu(\phi)= 0.
\end{equation}
In fact,
\begin{eqnarray*}
\int_{-\infty}^{+\infty} \bigl|\phi\bigl(e^{-x}(u-w_0)
\bigr) -\phi\bigl(e^{-x}u\bigr)\bigr|\nu(du)&\leq& C \int
_0^\infty\bigl|e^{-x}w_0\bigr|
\mathbf{1}_{[{\mathrm
{e}}^xm,{\mathrm{e}}^x(M+w_0)]} \nu(du)
\\
&\leq& C \bigl|e^{-x}w_0\bigr|\log\bigl({\mathrm{e}}^x(M+w_0)
\bigr),
\end{eqnarray*}
when $\operatorname{supp}(\phi)\subset[m,M]$. Summarizing, translation
of the
invariant measure does not change the problem we study, nor our
assumptions. Existence of a
corresponding $w_0$ is provided by the following lemma, whose proof
will be given in \hyperref[appendix]{Appendix}.

%
\begin{lem}
\label{lemtranslation}
There exists $w_0>0$ such that for all $\phi\in C_C^1(\R_+^*)$ the
functions $f_\phi$ and $g_\phi$ satisfy the hypotheses of Proposition
\ref{lemlim-poiss-eq}.
\end{lem}

Now we are ready to prove our main result.

\begin{pf*}{Proof of Theorem~\ref{mthm}}
We claim that $\int g_\phi(y) \,dy=0$. In fact for all $y$ we can apply
Corollary~\ref{corSV}
\[
\lim_{x\to+\infty}\frac{f_\phi(x+y)}{f_\phi(x)}=\lim_{x\to
+\infty}
\frac
{\delta_{{\mathrm{e}}^{-(x+y)}}*\nu_0(\phi)}{{\mathrm{e}}^{-x}*\nu
_0(\phi)}=1;
\]
thus, since $\E(S_t)$ is finite, by dominated convergence $\E(f_\phi
(x+S_t)/f_\phi(x))$ also converges to 1. Fix $\varepsilon>0$, then
there exists $x_\varepsilon$ such that for all $x\geq x_\varepsilon$
\[
\biggl|\E\bigl[f_\phi(x+S_t) \bigr]-f_\phi(x)+
\frac{1}{\E S_t}\int g_\phi(y)\,dy \biggr| <\varepsilon
\]
and
\[
\biggl|
\frac{\E[f_\phi(x+S_t)]}{f_\phi(x)}-1 \biggr| <\varepsilon.
\]
Therefore, $f_\phi(x)\geq|\int g_\phi(y) \,dy|/(\varepsilon\E S_t) -1$.
Since by Lemma~\ref{lemtranslation},
\[
\int_{-\infty}^xf_\phi(y) \,dy<
C(1+x),
\]
for all $x>x_\varepsilon>0$
\[
C(1+x)\geq\int_{x_\varepsilon}^xf_\phi(y) \,dy
\geq\biggl(\frac
{\llvert\int g_\phi(y) \,dy\rrvert}{\varepsilon\E S_t} -1 \biggr)
(x-x_\varepsilon).
\]
That is,
\[
\biggl|\int_\R g_\phi(y) \,dy \biggr|\leq\varepsilon\E
S_t \biggl(\liminf_{x\to+\infty} \frac{C (1+x)}{x-x_\varepsilon} +1
\biggr)=\varepsilon\E S_t(C +1).
\]
Letting $\varepsilon\searrow0$, we conclude.

In view of Corollary~\ref{corSV}, the quotient $f_\phi(x+y)/f_\phi
(x)$ is uniformly dominated by $1 +S_t$ for $x>0$ and $0<y <S_t$, thus
\[
\lim_{x\to\infty} \int_0^{S_t}
\frac{ f_\phi(x+y)}{ f_\phi(x)} \,dy = \int_0^{S_t} 1 \,dy =
S_t. \qquad\mbox{ $\P$ a.s.}
\]
By Fatou's lemma,
%
%
\begin{equation}
\label{eq-fatou} \qquad\liminf_{x\to\infty}\E\biggl[ \int
_0^{S_t} \frac{ f_\phi(x+y)}{
f_\phi(x)} \,dy \biggr]\geq\E
\biggl[\liminf_{x\to\infty} \int_0^{S_t}
\frac{
f_\phi
(x+y)}{ f_\phi(x)} \,dy \biggr]=\E[S_t].
\end{equation}
Therefore, by Proposition~\ref{lemlim-poiss-eq},
\begin{eqnarray*}
\limsup_{x\to\infty} f_\phi(x) &=& \limsup
_{x\to\infty} \frac{\E[\int_0^{S_t} f_\phi(x+y) \,dy ]
}{\E[\int_0^{S_t} (f_\phi(x+y)/{ f_\phi(x)}) \,dy ] }
\\
&\leq&\frac
{1}{\E[S_{l}]\E[S_t]}\int
_\R g_\phi(x)x\,dx. %
\end{eqnarray*}
In particular, this proves that $f_\phi(x)$ is bounded above. Since by
Corollary~\ref{corSV}, we already know that $f_\phi(x)$ is bounded
below, $\int_0^{S_t} \frac{ f_{\phi_i}(x+y)}{ f_{\phi}(x)} \,dy< CS_t$.
This\vspace*{1pt} allow to use the dominated convergence theorem instead of Fatou's
lemma in (\ref{eq-fatou}) and to replace the inferior limit with the
real limit and the inequality with the equality. Thus, we have
\begin{eqnarray*}
\lim_{x\to\infty} f_{\phi}(x) &=& \lim_{x\to\infty}
\frac{\E[\int_0^{S_t} f_{\phi}(x+y) \,dy ]
}{\E[\int_0^{S_t}( f_{\phi}(x+y)/ f_{\phi}(x))
\,dy ] }
\\
& =& \frac{1}{\E[S_{l}]\E[S_t]}\int_\R
g_{\phi}(x)x\,dx
=\frac{1}{\sigma
^2}\int_\R
g_{\phi}(x)x\,dx, %
\end{eqnarray*}
where $\sigma^2 = \int x^2 \bar\mu(dx)$ (see \cite{F} for the proof
that $\E[S_{l}]\E[S_t]=\sigma^2$).

To conclude, take a nonzero nonnegative function $\Phi\in
C_c^1(0,+\infty)$. We have proved that the following limit exists:
\[
\lim_{z\to+\infty}\delta_{z^{-1}}*\nu(\Phi)=\lim
_{x\to+\infty
}f_\Phi(x)=C
\]
and by Corollary~\ref{corSV} the constant $C$ is strictly positive.
The same corollary also implies that for all $\phi\in C_c(0,+\infty)$
\begin{eqnarray*}
\lim_{z\to+\infty}\delta_{z^{-1}}*\nu(\phi)&=&\lim
_{z\to+\infty
}\frac{\delta
_{z^{-1}}*\nu(\phi)}{\delta_{z^{-1}}*\nu(\Phi)} \lim_{z\to
+\infty
}\delta
_{z^{-1}}*\nu(\Phi)
\\
&=& \frac{C}{\int_\R\Phi(a)((da)/a)}\int_\R\phi
(a)\frac{da}{a}.
\end{eqnarray*}\upqed
\end{pf*}

\section{Uniqueness of the invariant measure}\label{sec5}
\label{secunique}
\mbox{}

\begin{pf*}{Proof of Theorem~\ref{teouniq}}
Notice first that for any compact set $K$
\begin{eqnarray*}
\lim_{n\to\infty} {\mathbf1}_K\bigl(X_n^y
\bigr)\bigl|X_n^y - X_n^{y'}\bigr| &\le&
\bigl|y-y'\bigr| \limsup_{n\to\infty} A_1\cdots
A_n {\mathbf1}_K\bigl(X_n^y\bigr)
\\
&=& \bigl|y-y'\bigr| \limsup_{n\to\infty}\frac{ X_n^y {\mathbf1}_K(X_n^y)} {
X_n^y/{(A_1\cdots A_n )}}
\\
&\leq& \limsup_{n\to\infty}\frac{ C(K)} {
{X_n^y}/({A_1\cdots A_n })}.
\end{eqnarray*}
%
Thus, it is sufficient to prove that
\[
\lim_{n\to\infty}\frac{X_n^y}{A_1\cdots A_n}=+\infty.
\]
Notice that the sequence $\frac{X_n^y}{A_1\cdots A_n }$ in
nondecreasing. Indeed, since
$\Psi_n(X_{n-1}^y) \ge A_n X_{n-1}^y$,
\[
\frac{X_n^y}{A_1\cdots A_n }=\frac{\Psi_n(X_{n-1}^y)}{A_1\cdots A_n
}\geq\frac{X_{n-1}^y}{A_1\cdots A_{n-1} }.
\]
Therefore, it is enough to justify that for arbitrary large fixed $M>0$
the sequence is a.s. at least once greater than $M$. Let
\[
U_{\beta, \gamma}:= \bigl\{\Psi\in\mathfrak{F}| \Psi[0,+\infty
)\subseteq[
\beta,+\infty)\mbox{ and } A(\Psi)<\gamma\bigr\}
\]
and
\[
V_\alpha:= \bigl\{\Psi\in\mathfrak{F}| A(\Psi)<\alpha\bigr\}.
\]
In view of our hypotheses, there exist $\alpha<1$, $\beta>0$, and
$\gamma$ such that these two sets have positive probability. For a
fixed $x_0$, take $N>0$ such that $\alpha^{N-1} M \gamma x_0 < \beta$
and let $\psi_0=\psi_1\psi_2$ with $\psi_1\in U_{\beta,\gamma}$ and
$\psi_2\in V_\alpha^{N-1}$. We claim that
%
%
\begin{equation}
\label{eq-psi-dil} \frac{\psi_0(x)}{A(\psi_0) x}>M\qquad\mbox{for all }
0\leq x\leq x_0.
\end{equation}
In fact,
\[
\psi_0(x)= \psi_1\bigl(\psi_2(x)\bigr)\geq
\beta> M \bigl(\gamma\alpha^{N-1} x_0\bigr)>M A(
\psi_1) A(\psi_2) x>MA(\psi_0)x. %
\]

Observe that since $X_n^y$ is recurrent, there exists $x_0>1$ such that
$\P[0\leq X_n^y <x_0\mbox{ i.o.}]=1$ for every $y\geq0$.
Let us fix $y$, $x_0$ and define a sequence $T_k$ of hitting times of $[0,x_0]$
%
\begin{eqnarray*}
T_0&=& 0,
\\
T_k &=& \inf\bigl\{ n>T_{k-1}+N\dvtx X_n^y
<x_0 \bigr\}.
\end{eqnarray*}
By recurrence, all $T_k$ are almost surely finite.
Let $\Psi_i^j:=\Psi_j\circ\cdots\circ\Psi_{i+1}$, then $\{\Psi
_{T_k}^{T_k+N}\}$ is a sequence of i.i.d. random transformations
distributed as $\mu^N$. Since $\mu^N(U_{\beta,\gamma}V_\alpha^{N-1})>0$
there exists almost surely $k_0$ such that $\Psi
_{T_{k_0}}^{T_{k_0}+N}\in U_{\beta,\gamma}V_\alpha^{N-1}$. Then, by
(\ref{eq-psi-dil}), we have
\[
\frac{X_{T_{k_0}+N}^y}{A_1\cdots A_{T_{k_0}+N}}=\frac{\Psi
_{T_{k_0}}^{T_{k_0}+N}(X_{T_{k_0}}^y)}{A_1\cdots A_{T_{k_0}+N} }\geq
\frac{\Psi
_{T_{k_0}}^{T_{k_0}+N}(X_{T_{k_0}}^y)x_0}{A_{T_{k_0}+1}\cdots
A_{T_{k_0}+N} X_{T_{k_0}}^y}>M.
\]\upqed
\end{pf*}

\section{Examples}\label{sec6}
\label{secexamples}

In this section, we present some of the more significant classes of
stochastic dynamical system to which the results of the previous
sections apply.

\subsection{The random difference equation}\label{sec6.1}
The first example is naturally the SDS induced by random affinities,
that is $\Psi_n(x) = A_nx+B_n$, for a random pair $(B_n,A_n) \in\R
\times\R^*_+$. Then $X_n^x$ is given by formula (\ref
{eqaffinerecursion}). This process is called the random difference
equation or
the affine recursion. It is well known that under the assumptions of
Theorem~\ref{mthm} this process is recurrent and locally contractive,
thus it possesses a unique invariant Radon measure $\nu$; see \cite
{BBE}. Behavior of this measure at infinity was studied previously in
\cite{Bu,BBD,BBD2} under a number of additional strong hypotheses.
Theorem~\ref{mthm} provides an optimal result, in the sense that the
hypotheses implying existence and uniqueness of the invariant measure,
are sufficient also to deduce that this measure must behave at infinity
like $\frac{C \,dx}x$.

\subsection{Stochastic recursions with unique invariant measure}\label{sec6.2} Our
results can also be applied to a more general class of stochastic
recursions that behave at infinity as $Ax$ [i.e., $\Phi(x)\sim Ax$ for
large $x$]. In the contracting case ($\E[\log A]<0$), those recursions
were studied by Goldie \cite{Go} (see also Mirek \cite{M}, who
described this class of recursions in general settings, including more
examples). Just to give some concrete examples let us mention that our
results are valid (under rather obvious and easy to formulate
assumptions) for the following examples:
\begin{itemize}
\item$\Psi_{1,n}(x) = \max\{A_nx,B_n\}+C_n$, for $A_n,B_n,C_n>0$.
\item$\Psi_{2,n}(x) = \sqrt{A_n^2 x^2 + B_nx + C_n}$, for $A_n,B_n,C_n
> 0$ and $\Delta= B^2-4A^2C \le0$.
\end{itemize}
%
In both cases above, the mappings $\Psi_{i,n}$ are Lipschitz with the
Lipschitz\break coefficient equal to $A$. This is obvious for the first
example. For the second one, denote $x_0 = -\frac{B}{2A^2}$, $D =
-\frac
{\Delta}{4A^2}$. Observe that since $\Psi_{2,n}(x)=\break \sqrt
{A^2(x-x_0)^2+D}$, its derivative
\[
\Psi'_{2,n}(x)=\frac{A^2(x-x_0)}{\sqrt{A^2(x-x_0)^2+D}}=A\frac
{1}{\sqrt{1+ (D/ (A^2(x-x_0)^2))}}
\nearrow A
\]
is an increasing function that tends to $A$.
Hence, under appropriate moment assumptions, the SDS on $\R_+$
generated by the random functions defined above satisfies assumptions
of both Theorems~\ref{mthm}~and~\ref{teouniq}. Therefore, the
corresponding random process possesses a unique invariant measure,
which behaves at infinity like $\frac{C\,dx}{x}$.

If we do not suppose $\Delta=B^2-4A^2C \le0$, then $\Psi_{2,n}$ are
still asymptotically linear functions to which Theorem~\ref{mthm}
applies, but we cannot prove uniqueness of an invariant measure.

\subsection{Random automorphisms of the interval $[0,1]$}\label{sec6.3}
\label{secinterval}
SDSs acting on the real line after conjugating by an appropriate
function can be seen as random automorphisms of the interval $[0,1]$
fixing the end points. Our key property (\ref{AL}) is translated in this
setting into requiring that the automorphisms ``reflect'' at the same way
in 0 and in 1, in the sense that the derivative in these two points has
to be the same. The $B$ term is then related to the term of order two
at these end points [or order $2-\alpha$, if we conjugate a SDS that
satisfy (\ref{ALal})]. More precisely, we have the following.

%
\begin{cor}
Consider a SDS on $[0,1]$ defined by random functions $\phi\in
C([0,1])$ fixing $0$ and $1$, differentiable at the extremities of the
interval and such that
\[
\phi'(0)=\phi'(1)=:a_\phi.
\]
Let
\begin{eqnarray*}
\beta^0_1 &=& \inf_{u\in[0,1/2]}
\bigl(1- \phi(u)\bigr)>0,\qquad\beta^0_2 = \inf
_{u \in[0,1/2]}\frac{ \phi(u)}{u}>0,
\\
\beta^0_3 &=& \sup_{u\in[0,1/2]} \biggl\llvert\frac{\phi(u)-a_\phi
u}{u^2}\biggr\rrvert<\infty,
\\
\beta^1_1 &=& \inf_{u\in[1/2,1]} \phi(u)>0,\qquad
\beta^1_2 = \inf_{u \in[1/2,1]}
\frac{1-\phi(u)}{1-u}>0,
\\
\beta^1_3 &=& \sup
_{u\in[1/2,1]} \biggl\llvert\frac{\phi(u)-1-a_\phi
(u-1)}{(u-1)^2}\biggr\rrvert<\infty.
\end{eqnarray*}
Suppose that $\E[|\log a_\phi|^{2+\varepsilon} ]<\infty
$, $\E
[|\log\beta_k^i|^ {2+\varepsilon} ]<\infty$, for some
$\varepsilon>0$, all
$i,k$, and that $\E[\log a_\phi]=0$. Then the SDS on $[0,1]$ is
conjugated to an asymptotically linear SDS on $\R$ that satisfy the
hypotheses of Theorem~\ref{mthm}. Therefore, there exists at least one
invariant Radon measure $\tilde\nu$ on $(0,1)$ and for every such
a measure $\tilde{\nu}$, which charges a neighborhood of 0, there
exists a strictly positive constant $C$ such that for all $0<a<b<1$
\[
\lim_{z\to+\infty}\tilde{\nu}(a/z,b/z)=C \log b/a.
\]
\end{cor}

\begin{pf}
Let
\[
r(u)=-\frac{1}{u}+\frac{1}{1-u}
\]
be a diffeomorphism of $(0,1)$ onto $\R$. In the technical Lemma~\ref{lemALinterval}, whose proof is postponed to \hyperref[appendix]{Appendix}, we prove that
the conjugated function $\Psi_\phi=r \circ\phi\circ r^{-1}$ satisfy
(\ref{AL}) for $A(\Psi_\phi)=1/{a_{\phi}}$ and
\[
B(\Psi_\phi)< C_r \biggl(\frac{(1+a_\phi+\beta_3^0)}{a_{\phi}
\beta^0_2} +
\frac{1}{\beta^0_1} + \frac{(1+a_\phi+\beta_3^1)}{a_{\phi}
\beta^1_2} + \frac{1}{\beta^1_1} \biggr),
\]
where $C_r$ depends only on the function $r$.
Thus, under the hypotheses of the corollary, the conjugated SDS
satisfies the assumptions of our main theorem.

Let $\tilde\mu$ be the law of $\phi$ and $\mu=r*\tilde\mu*
r^{-1}$ be the law of the conjugated SDS on~$\R$. Then $\nu$ is a
$\mu
$-invariant Radon measure on $\R$ if and only if $\tilde\nu
=r^{-1}*\nu$ is a \mbox{$\tilde\mu$-}invariant Radon measure on $(0,1)$.
Then by Theorem~\ref{mthm} and since $|r(u)+1/u|<2$ for $0<u<1/2$,
\begin{eqnarray*}
\biggl|\tilde\nu\biggl(\frac{a}z,\frac{b}z \biggr)-
\nu\biggl(-\frac{z}a,-\frac{z}b \biggr) \biggr| &=& \biggl|\nu\biggl(r
\biggl(\frac{a}z \biggr),r \biggl(\frac{b}z \biggr) \biggr)-\nu
\biggl(-\frac{z}a,-\frac{z}b \biggr) \biggr|
\\
&\leq& \nu\biggl(-\frac{z}a-2, -\frac{z}a+2 \biggr)+\nu
\biggl(-\frac{z}b-2, -\frac{z}b+2 \biggr) \to0
\end{eqnarray*}
for $z\to+\infty$. Thus,
\[
\lim_{z\to+\infty}\tilde\nu(a/z,b/z )=\lim_{z\to
+\infty}
\nu\biggl(-\frac{z}a,-\frac{z}b \biggr)= C \log b/a.
\]\upqed
\end{pf}

\subsection{Additive Markov processes and power functions}\label{sec6.4}
When an asymptotically linear SDS is conjugated by a homeomorphism of
the real line which behaves as the exponential at infinity, it is
transformed into a SDS that is asymptotically a translation or, by the
reversed conjugation, a power function.

More precisely, consider a SDS generated by functions $\phi$ such that
%
%
\begin{equation}
\label{eqasympt-trans} \bigl|\phi(x)-x+\operatorname{sign}(x)u_\phi\bigr|\leq
v_\phi{\mathrm{e}}^{-|x|}
\end{equation}
for some constants $u_\phi$ and $v_\phi$. This class contains mappings
of $[0,\infty)$ that are equal to translations outside a bounded set,
that is, a Markov additive process as defined in Aldous (\cite{Ald89}, Sections~C11,~C33).
Let $s$ be a continuous bijection of $\R$ such that
\[
s(x)={\mathrm{e}}^x\qquad\mbox{for } x>1 \quad\mbox{and}\quad s(x)=-{
\mathrm{e}}^{-x}\qquad\mbox{for } x<-1.
\]
Then the SDS generated by $\psi_\phi(x)=s\circ\phi\circ s^{-1}$
satisfies hypothesis (\ref{AL}) with $A(\psi_\phi)={\mathrm{e}}^{-u_\phi
}$. Hence,
under moment conditions that can be obtained with standard
calculations, if $\E(u_\phi)=0$ there exists an invariant measure which
behaves at infinity as the Lebesgue measure $dx$, that is,
\[
\lim_{z\to+\infty}\tilde\nu(\alpha+z,\beta+z)=C(\beta-\alpha)
\]
for every measure of unbounded support, some constant $C>0$ and all
$\beta>\alpha$.

In a similar way, a SDS generated by function $\phi$ such that
\[
|x|^a \cdot\operatorname{sign}(x) {\mathrm{e}}^{-b_1\log(|x|+2)^\alpha}
\leq\phi
(x)\leq|x|^a\cdot\operatorname{sign}(x) {\mathrm{e}}^{+b_1\log
(|x|+2)^\alpha}
\]
for some $\alpha$ is associated to an $\alpha$-asymptotically linear
system by the reverse conjugation $\psi_\phi(x)=s^{-1}\circ\phi
\circ
s$ and $A(\psi_\phi)=a$. Thus, if $\E(\log a)=0$ and some moments are
finite, for any invariant measure $\tilde{\nu}$, whose support in
unbounded on the positive half-line, there exists a strictly positive
constant $C$ such that for all $1< \alpha<\beta$
\[
\lim_{z\to+\infty} \tilde{\nu}\bigl(\alpha^z,
\beta^z\bigr)=C \log\frac{\log
\beta}{\log\alpha}.
\]

\subsection{Population of Galton--Watson tree with random reproduction law}\label{sec6.5}
Consider the following model of reproduction of a population. Let $\{
\rho_\omega|\omega\in\Omega\}$ be the set of probability measures on
the set of natural numbers $\N$ and $\lambda(d\omega)$ be a probability
law on $\Omega$. At each generation, a law of reproduction $\rho
_\omega
$ is chosen according to $\lambda(d\omega)$ and each individual $j$ is
replaced by $r_j$ offsprings, $r_j$ chosen according to the law $\rho
_\omega$ and independently from the other individuals. To prevent the
extinction of the population, a random immigration $i_\omega$ it added
to the population. More formally, if the population consists of $x\in
\N
$ individuals, the population of the following generation is
\[
\psi_{\omega,\mathbf{r}}(x)=i_\omega+\sum_{j=1}^xr_j,
\]
where the reproduction law $\omega\in\Omega$ is chosen according to
$\lambda(d\omega)$, $\mathbf{r}=\{r_j\}_j$ are i.i.d. of law $\rho
_\omega
$ and $i_\omega$ is a random variable. If every generation is
independent from the previous one, then the evolution of the population
is a SDS on $\mathcal{R}=\N$ of law $\mu(d\psi)=\otimes\rho
_\omega
(d\mathbf{r})\lambda(d\omega)$. If $\E r_1^2<\infty$, the law of
iterated logarithm
proves that the $\psi_{\omega,\mathbf{r}}$ are $\mu$-almost surely
$\alpha$-asymptotically linear with an error of order $x^\alpha$ for
all $\alpha>1/2$ and
\[
A(\psi)=A_\omega=\int_\N r \rho_\omega(dr)=
\mbox{average number of offspring per individual for $\rho_\omega$}.
\]
Unlike the classical Galton--Watson process, in our context $A_\omega$
is not constant, but varies from one generation to another. The key
parameter, that decides whether the system is recurrent, is $\E(\log
A_\omega)=\int\log A_\omega\lambda(d\omega)$. To apply Theorem
\ref
{mthm}, we need to control moments of $B(\psi)$. The details are stated
in the following lemma. Our estimates are fairly rough and the
hypotheses could be probably improved, but this go beyond the purpose
of our paper.

%
\begin{lemma} Suppose $\E(r_1^4)=\int_\Omega\int_\N r_1^4 \rho
_\omega
(dr)\lambda(d\omega)<\infty$. Let $\alpha>3/4$ and
\[
B(\psi)=B_{\omega,\mathbf{r}}=\sup_{x\in\N}\frac{\llvert\psi
_{\omega,\mathbf{r}}(x)-A_\omega x\rrvert}{x^\alpha+1},
\]
then there exists a finite constant $C_\alpha$ that only depends on
$\alpha$ such that
\[
\E\bigl(\bigl(\log^+B(\psi)\bigr)^{2+\epsilon}\bigr) \leq C_\alpha
\bigl(1+ \E\bigl(\bigl(\log^+ i_\omega\bigr)^{2+\epsilon}\bigr)+ \E
\bigl(r_1^4\bigr) \bigr).
\]
\end{lemma}

\begin{pf}
Observe first
\[
\frac{\llvert\psi_{\omega,\mathbf{r}}(x)-A_\omega x\rrvert}{x^\alpha
+1}\leq i_\omega+\frac{ |\sum_{j=1}^x (r_j-A_\omega)
|}{x^\alpha+1}.
\]
Thus,
\[
\bigl(\log^+B(\psi)\bigr)^{2+\epsilon}\leq C \biggl(\bigl(C+\log^+
i_\omega\bigr)^{2+\epsilon
}+ \sup_{x\in\N} \biggl(
\log^+ \frac{ |\sum_{j=1}^x
(r_j-A_\omega)
|}{x^\alpha+1} \biggr)^{2+\epsilon} \biggr). %
\]
Let $y_j:=r_j-A_\omega$ be centered random variables. For a fixed
reproduction law~$\omega$, denote by $\P_\omega$ the quenched
probability. Since under $\P_\omega$ the variables $y_j$ are
independent, $\E_\omega(y_{j_1}y_{j_2}y_{j_3}y_{j_4})=0$ if there
exists an index $j_k$ that is different from all the others. Then
standard calculus shows that
\begin{eqnarray*}
\E_\omega\Biggl[\sum_{j=1}^x
y_j \Biggr]^4 &=& \sum_{j_1,j_2,j_3,j_4=1}^x
\E_\omega(y_{j_1}y_{j_2}y_{j_3}y_{j_4})
\\
&=& x \E_\omega\bigl(y_1^4\bigr)+3x(x-1)
\bigl(\E_\omega\bigl[y_1^2\bigr]
\bigr)^2\leq4x^2\E_\omega\bigl(y_1^4
\bigr).
\end{eqnarray*}
Finally, since $\alpha>3/4$, we have
\begin{eqnarray*}
&& \E\biggl(\sup_{x\in\N} \biggl(\log^+
\frac{ |\sum_{j=1}^x
(r_j-A_\omega
) |}{x^\alpha+1} \biggr)^{2+\epsilon} \biggr)
\\
&&\qquad \leq C \E\biggl(\sum
_{x\in\N} \biggl( \frac{ |\sum_{j=1}^x
y_j|}{x^\alpha+1} \biggr)^{4}
\biggr)
= C \E\biggl(\sum_{x\in\N} \frac{\E_\omega[ |\sum_{j=1}^x
y_j |^{4} ]}{(x^\alpha+1)^4}
\biggr)
\\
&&\qquad \leq C \E\biggl(\sum_{x\in\N} \frac{4x^2\E_\omega
(y_1^4)}{(x^\alpha
+1)^4}\biggr)=C\sum_{x\in\N} \frac{4x^2\E(y_1^4)}{(x^\alpha
+1)^4}<\infty.
\end{eqnarray*}\upqed
\end{pf}

\subsection{Reflected random walk in critical case}\label{sec6.6}
\label{secreflectedproof}
The reflected random walk
\[
Y_n^x = \bigl| Y_{n-1}^x -
u_n \bigr|,
\]
is an example of an asymptotic translation for which (\ref
{eqasympt-trans}) holds. Thus, we can apply our main Theorem~\ref{mthm}.
However, in this case the same results hold under weaker hypotheses and
a much more direct proof. We give here the proof of Theorem~\ref{mthm2}, stated in the \hyperref[sec1]{Introduction}.

\begin{pf*}{Proof of Theorem~\ref{mthm2}}
Define the upward ladder times of $S_n = \sum_{i=1}^n u_i$:%
%
\begin{eqnarray*}
t_0 &=& 0,
\\
t_{k+1} &=& \inf\{n>t_k\dvtx S_n \geq
S_{t_k}\}
\end{eqnarray*}
and put $\bar u_k = S_{t_{k}} - S_{t_{k-1}}$. Then $\{\bar u_k\}$ is a
sequence of i.i.d. random
variables and every $\bar u_k$ is equal in distribution to
$S_{t_1}$. We
define reflected random
walk for $\{\bar u_k\}$:
\begin{eqnarray*}
\overline Y_0^x &=& x,
\\
\overline Y^x_{k+1} &=& \bigl| \overline Y_k^x
- \bar u_{k+1} \bigr|,
\end{eqnarray*}
then $\overline Y^x_{k} = Y^x_{t_k}$.
In view of the result of Chow and Lai \cite{CL}, $\E({\bar u_k})^{{1}/2}<\infty$
and this is sufficient for existence of a unique invariant probability
measure $\nu_t$ of the process $\{\overline Y_k^x\}$ (see \cite{PW1}
for more
details).
Let us define the measure
\[
\nu_0(f) = \int_{\R_+}\E\Biggl[\sum
_{n=0}^{t-1} f\bigl(Y_n^x
\bigr) \Biggr]\nu_t(dx).
\]
Notice first that this is a Radon measure. Indeed, define ${l_i}=\inf\{
n> l_{i-1}\dvtx S_n < S_{l_{i-1}}\}$. Since $\E(u_1^-)^2<\infty$,
$-\infty<\E
S_{l} <0$ (see \cite{CL}). Take any $f\in C_C(\R_+)$, then
by the duality lemma \cite{F}
%
%
\begin{eqnarray}\label{eqduality}
\nu_0(f) &=& \int_{\R_+}\E\Biggl[
\sum_{n=0}^{t-1} f(x-S_n) \Biggr]
\nu_t(dx)
\nonumber\\[-8pt]\\[-8pt]
&=& \int_{\R_+}\E\Biggl[\sum
_{n=0}^{\infty} f(x-S_{l_n}) \Biggr]
\nu_t(dx).\nonumber
\end{eqnarray}
By the renewal theorem,
\begin{eqnarray*}
\E\Biggl[\sum_{n=0}^{\infty}
f(x-S_{l_n}) \Biggr] &\le& C \E[\#n\dvtx \alpha< x-S_{l_n} <
\beta] \le C|\beta-\alpha|,
\end{eqnarray*}
therefore, $\nu_0(f)$ is finite, and thus $\nu_0$ is a Radon measure.

Next, since $\mu_t *\nu_t = \nu_t$,
we have
\begin{eqnarray*}
\mu*\nu_0(f) &=& \int_{\R_+} \int
_{\R}\E\Biggl[\sum_{n=0}^{t-1}
f\bigl(Y_n^{|x-y|}\bigr) \Biggr]\mu(dy)\nu_t(dx)
\\
&=& \int_{\R_+} \E\Biggl[\sum_{n=1}^{t}
f\bigl(Y_{n}^x\bigr) \Biggr]\nu_t(dx)=
\nu_0(f).
\end{eqnarray*}
Therefore, $\nu_0$ is a $\mu$ invariant Radon measure, so $\nu_0 =
C\nu
$ and without any loss of generality we may assume $\nu=\nu_0$.

Finally, by (\ref{eqduality}), the Lebesgue theorem and the renewal theorem
\begin{eqnarray*}
\lim_{z\to\infty} \int_{\R_+}f(u-z)\nu(du) &=& \lim
_{z\to\infty} \int_{\R_+}\E\Biggl[\sum
_{n=0}^{\infty} f(x-S_{l_n}-z) \Biggr]
\nu_t(dx)
\\
&=& \frac{1}{-\E S_{l}}\int_{\R_+}f(x)\,dx
\end{eqnarray*}
and the theorem is proved.
\end{pf*}

\setcounter{equation}{0}
\appendix\label{appendix}
\section*{Appendix: Proofs of technical results}
In this appendix, we give the postponed proofs of the technical results
stated in Lemma~\ref{lemalphatozero},
Proposition~\ref{lemlim-poiss-eq}, Lemma~\ref{lemtranslation}. At
the end, we formulate and prove
Lemma~\ref{lemALinterval}, which is used in Section~\ref{secinterval}.

\begin{pf*}{Proof of Lemma~\ref{lemalphatozero}}
We will prove the result only for positive $x$, since for negative
values of $x$ the same argument is valid just by conjugating with the
map $x\mapsto-x$.

Suppose first $x\geq1$. We have
\begin{eqnarray*}
&& r \bigl(A_{\alpha}r^{-1}(x)-B_\alpha
\bigl(1+\bigl|r^{-1}(x)\bigr|^\alpha\bigr) \bigr)
\\
&&\qquad \leq
\psi_r(x)\leq r \bigl(A_{\alpha}r^{-1}(x)+B_\alpha
\bigl(1+\bigl|r^{-1}(x)\bigr|^\alpha\bigr) \bigr),
\\
&& r \bigl(A_{\alpha} x^{1/(1-\alpha)}-B_\alpha\bigl(1+
x^{\alpha/(1-\alpha)}\bigr) \bigr)
\\
&&\qquad  \leq\psi_r(x)\leq r
\bigl(A_{\alpha} x^{1/(1-\alpha)}+B_\alpha\bigl(1+ x^{\alpha/(1-\alpha)}
\bigr) \bigr),
\\
&& r \bigl(A_{\alpha} x^{1/(1-\alpha)}-B_\alpha c_\alpha
x^{\alpha/(1-\alpha)} \bigr)
\\
&&\qquad  \leq\psi_r(x)\leq r \bigl(A_{\alpha}
x^{1/(1-\alpha)}+B_\alpha c_\alpha x^{\alpha/(1-\alpha)}
\bigr),
\end{eqnarray*}
where $c_\alpha$ only depends on $\alpha$. Suppose further
$x>c_\alpha
B_\alpha/A_{\alpha}$, then $A_{\alpha} x^{1/(1-\alpha)}-B_\alpha
c_\alpha x^{\alpha/(1-\alpha)}>0$ and
\begin{eqnarray*}
&& \bigl(A_{\alpha} x^{1/(1-\alpha)}-B_\alpha c_\alpha
x^{\alpha/(1-\alpha)} \bigr)^{1-\alpha}
\\
&&\qquad  \leq\psi_r(x)\leq
\bigl(A_{\alpha} x^{1/(1-\alpha)}+B_\alpha c_\alpha
x^{\alpha/(1-\alpha)} \bigr)^{1-\alpha},
\\
&& A_{\alpha}^{1-\alpha} x^{(1-\alpha)/(1-\alpha)} - A_{\alpha
}^{-\alpha}
x^{-\alpha/(1-\alpha)} B_\alpha c_\alpha x^{\alpha/(1-\alpha)}
\\
&&\qquad  \leq
\psi_r(x)\leq A_{\alpha}^{1-\alpha} x^{(1-\alpha)/(1-\alpha)}
+(1-\alpha) A_{\alpha}^{-\alpha} x^{-\alpha/(1-\alpha)} B_\alpha
c_\alpha x^{\alpha/(1-\alpha)},
\\
&& A_{\alpha}^{1-\alpha} x - A_{\alpha}^{-\alpha}
B_\alpha c_\alpha
\\
&&\qquad  \leq\psi_r(x)\leq
A_{\alpha}^{1-\alpha} x +(1-\alpha) A_{\alpha
}^{-\alpha}
B_\alpha c_\alpha
\end{eqnarray*}
since for $x_0>0$ and $h>0$, by concavity $ (x_0+h)^{1-\alpha}\leq
x_0^{1-\alpha}+(1-\alpha)x_0^{-\alpha}h$ and $(x_0-h)^{1-\alpha
}\geq
x_0^{1-\alpha}-x_0^{-\alpha}h$. Hence, we proved the lemma for
$x> \max\{
1,\break c_\alpha B_\alpha/A_\alpha\}$
Now, for $x<1$
\begin{eqnarray*}
&& r \bigl(A_{\alpha} x^{1/(1-\alpha)}-B_\alpha\bigl(1+
x^{\alpha/(1-\alpha)}\bigr) \bigr)
\\
&&\qquad  \leq\psi_r(x)\leq r
\bigl(A_{\alpha} x^{1/(1-\alpha)}+B_\alpha\bigl(1+ x^{\alpha/(1-\alpha)}
\bigr) \bigr),
\\
&& - (2 B_\alpha)^{1-\alpha}
\\
&&\qquad  \leq \psi_r(x)\leq
(A_{\alpha}+ 2B_\alpha)^{1-\alpha}
\end{eqnarray*}
and for $ x \leq c_\alpha B_\alpha/A_{\alpha}$.
%
\begin{eqnarray*}
&& - \biggl(B_\alpha\biggl(1+ \biggl(\frac{c_\alpha B_\alpha}{A_{\alpha
}}
\biggr)^{\alpha/(1-\alpha)} \biggr) \biggr)^{1-\alpha}
\\
&&\qquad \leq\psi_r(x)
\leq\biggl(A_{\alpha} \biggl(\frac{c_\alpha B_\alpha
}{A_{\alpha
}} \biggr)^{1/(1-\alpha)}+B_\alpha
\biggl(1+ \biggl(\frac
{c_\alpha
B_\alpha}{A_{\alpha}} \biggr)^{\alpha/(1-\alpha)} \biggr)
\biggr)^{\alpha/(1-\alpha)}.
\end{eqnarray*}
%
Hence, the lemma follows.
\end{pf*}

\begin{pf*}{Proof of Proposition~\ref{lemlim-poiss-eq}}
\textit{Step} 1. Let $t_k$ and $l_k$ be the stopping times defined in
(\ref{stopping}).
%
Let $U_l$ be the potential of the random walk $S_{l_k}$ and let
\[
R(x):= \sum_{k=0}^\infty\E\bigl[ g(x+
S_{l_k}) \bigr]=U_l(\delta_x *_\R
g).
\]
Since the function $g$ is directly Riemann integrable and $-\infty<\E
S_{l}<0$, the function $R$ is well defined and finite for every $x$.
Notice also that by the duality lemma~\cite{F}
%
%
\begin{equation}
\label{eqdef-R} R(x) = \sum_{k=0}^\infty\E
\bigl[ g(x+ S_{l_k}) \bigr] = \E\Biggl[ \sum
_{k=0}^{t-1} g(x+ S_k) \Biggr].
\end{equation}

\textit{Step} 2. We claim that
%
%
\begin{equation}
\label{eqf-R} \E\bigl[ f(x+ S_t) \bigr]-f(x) = \E\Biggl[ \sum
_{k=0}^{t-1} g(x+ S_k)
\Biggr]=R(x).
\end{equation}
Indeed, the process $ f(x+S_n) - \sum_{k=0}^{n-1} g(x+S_k)$ forms a
martingale [for this purpose one just has to iterate the Poisson
equation (\ref{poissonequation})]. Then for any fixed $n$, $T\wedge n$
is a bounded stopping time, therefore, by the optional stopping time
theorem we have
\[
f(x) = \E\bigl[ f(x+ S_{t\wedge n}) \bigr] - \E\Biggl[ \sum
_{k=0}^{(t\wedge n)-1} g(x+ S_k) \Biggr].
\]
To justify that, we can let $n$ tend to infinity and change the order
of the limit
and the expected value to obtain (\ref{eqf-R}) observe that
\[
\E\bigl[ f(x+S_{t\wedge n} )\bigr] \le C\E\bigl[1+ (x+S_{t\wedge
n})^+
\bigr]\le C\E\bigl[1+ (x+S_{t})^+ \bigr] <\infty.
\]
The second term is uniformly dominated in $n$ by
\[
\E\Biggl[ \sum_{k=0}^{t-1} |g|(x+
S_k) \Biggr] = \sum_{k=0}^\infty
\E\bigl[ |g|(x+ S_{l_k}) \bigr]<\infty, %
\]
therefore converges to $R(x)$ when $n$ goes to infinity.

This proves that
\[
\E\bigl[ f(x+ S_t) \bigr]-f(x) = R(x) = U_l(
\delta_x*_\R g)
\]
and by the renewal theorem we obtain (\ref{eqlim-poiss-eq-non-cent}).

\textit{Step} 3. Let
\[
G(x):=\int_{-\infty}^x g (z)\,dz.
\]
If we suppose $\int g(x) \,dx=0$ then
\[
G(x) =\int^{+\infty}_{-\infty} g(z)\,dz-\int
^{\infty}_x g(z)\,dz=-\int^{+\infty}_x
g(z)\,dz. %
\]
Thus,
\[
\bigl|G(x)\bigr| \le\int_{-\infty}^x \bigl|g (z)\bigr|\,dz
\mathbf{1}_{(-\infty,0]}(x)+\int^{\infty
}_x \bigl|g(z)\bigr|\,dz
\mathbf{1}_{[0,+\infty)}(x) =: G_1(x)+G_2(x)
\]
and $G$ is directly Riemann integrable since functions $G_i$ are
monotone and integrable on their support:
\begin{eqnarray*}
\int_{-\infty}^0G_1(x) \,dx
&=& \int_{-\infty}^{+\infty} \int_{-\infty
}^{+\infty}
\mathbf{1}_{[z<x<0]} \bigl|g (z)\bigr|\,dx \,dz =\int_{-\infty}^{0}
\bigl|z g (z)\bigr| \,dz <\infty,
\\
\int^{+\infty}_0G_2(x) \,dx &=& \int
_{-\infty}^{+\infty} \int_{-\infty
}^{+\infty}
\mathbf{1}_{[z>x>0]} \bigl|g (z)\bigr|\,dx \,dz =\int^{+\infty}_{0}
\bigl|z g (z)\bigr| \,dz <\infty.
\end{eqnarray*}
Furthermore,
\begin{eqnarray*}
\int_{-\infty}^\infty G(x) \,dx &=& \int
_{-\infty}^{+\infty} \int_{-\infty
}^{+\infty}
\mathbf{1}_{[z<x<0]} g (z)\,dx \,dz
\\
&&{} - \int_{-\infty
}^{+\infty}
\int_{-\infty}^{+\infty} \mathbf{1}_{[z>x>0]} g (z)\,dx
\,dz
\\
&= &-\int^{+\infty}_{-\infty} z g (z)\,dz.
\end{eqnarray*}

\textit{Step} 4.
By the renewal theory, $U_l(\delta_x *_\R G)$ is well defined and by
Fubini's theorem
\begin{eqnarray*}
\int_{-\infty}^x R(z)\,dz &=& \int_{-\infty}^0
\int_{-\infty}^x g(z+u)\,dzU_l(du)
\\
&=&\int
_{-\infty}^0\int_{-\infty}^{x+u}
g(z)\,dzU_l(du)
\\
&=& U_l(\delta_x *_\R G).
\end{eqnarray*}
On the other hand,
\[
\int_{-\infty}^x R(z)\,dz = \E\biggl[ \int
_{-\infty}^x f(z+ S_t)\,dz - \int
_{-\infty}^x f(z)\,dz \biggr]=\E\biggl[ \int
_x^{x+S_t} f(z)\,dz \biggr].
\]
In fact, the two integrals above are finite because by our hypotheses
\begin{eqnarray*}
\int_{-\infty}^x \E\bigl[f(y+S_t)
\bigr]\,dy &=& \E\biggl[ \int_{-\infty}^{x+S_t} f(y)\,dy
\biggr]\leq C\E\bigl[1+(x+S_t)^+\bigr]
\end{eqnarray*}
and $\E S_t<\infty$ since $\bar{\mu}$ has moment of order
$2+\epsilon
$, see \cite{CL}. Thus, we proved
\[
\E\biggl[ \int_x^{x+S_t} f(z)\,dz \biggr] =
\delta_x *_\R U_{l}(G)
\]
and we can conclude using again the renewal theorem.
\end{pf*}

\begin{pf*}{Proof of Lemma~\ref{lemtranslation}}
\textit{Step} 1. Let $0<\gamma<1$, then the set of $v>0$ such that the
function $u\mapsto(u-v)^{-\gamma}$ is $\nu$-integrable on
$(v,+\infty
)$ is of full Lebesgue measure. In fact, for any interval $[a,b]\subset
(0,\infty)$,
\begin{eqnarray*}
&& \int_a^b \biggl(\int
_v^\infty(u-v)^{-\gamma}\nu(du) \biggr)\,dv
\\
&&\qquad =  \int_a^{2b} \biggl(\int_a^{u\wedge b}
(u-v)^{-\gamma} \,dv \biggr)\nu(du)+ \int_{2b}^\infty
\biggl(\int_a^{u\wedge b} (u-v)^{-\gamma} \,dv
\biggr)\nu(du)
\\
&&\qquad \leq \int_a^{2b} \biggl(\int
_{0} ^{2b-a} w^{-\gamma} \,dw \biggr)\nu(du)+\int
_{2b} ^\infty\biggl(\int_{u-b}
^{u-a} w^{-\gamma} \,dw \biggr)\nu(du)
\\
&&\qquad = C + \int_{2b}^\infty(u-b)^{-\gamma} (b-a)
\nu(du)<\infty.
\end{eqnarray*}
Take $w_0$, such that $\int_{w_0}^\infty(u-w_0)^{-\gamma}\nu
(du)<\infty
$, then
\[
f_\phi(x)=\int_{w_0}^\infty\phi
\bigl(e^{-x}(u-w_0)\bigr) \nu(du)\leq C \int
_{w_0}^\infty e^{\gamma x}(u-w_0)^{-\gamma}
\nu(du)\leq C e^{\gamma x},
\]
this gives good estimates of $f_{\phi}$ for negative $x$'s.

\textit{Step} 2.
Let $\operatorname{supp}(\phi)\subset[m,M]$. By Proposition~\ref{propboundnu}, the tail of $\nu$ is at most logarithmic, therefore for
$x\geq0$,
\[
f_\phi(x)\leq\nu\bigl(\bigl[e^xm+w_0,
e^xM+w_0\bigr]\bigr)\leq\nu\bigl(\bigl[e^xm,
e^x(M+w_0)\bigr]\bigr)\leq C (1+ x)
\]
and
\begin{eqnarray*}
\int_{-\infty}^x f_\phi(y) \,dy
&\leq& C \int_\mathbb{R}\int_{-\infty
}^\infty
\mathbf{1}_{[y<x]} \mathbf{1}_{[m,M]}\bigl({\mathrm{e}}^{-y}(u-w_0)
\bigr) \,dy \nu(du)
\\
&\leq& C \int_\mathbb{R}\mathbf{1}_{[w_0< u\leq{\mathrm{e}}^x
(M+w_0)]}\log
\frac
{M}{m}\nu(du)\leq C\bigl(1+ x^+\bigr).
\end{eqnarray*}
This proves (\ref{eqcondf}).

\textit{Step} 3.
We need to justify that $g_\phi=\bar\mu* f_\phi-f_\phi$ is directly
Riemann integrable, and moreover $\int_\R|x g(x)| \,dx < \infty$. We
recall first that, since $g$ is continuous, to prove that it is
directly Riemann integrable is sufficient to show that $|g|$ is
dominated on $(-\infty,0]$ (resp., on $[0,+\infty)$) by an integrable
nondecreasing (resp., nonincreasing) function. For $x<0$,
\begin{eqnarray*}
&& \bar\mu* f_\phi(x)
\\
&&\qquad = \int_{-\infty}^{+\infty}
f_\phi(x+y)\bar\mu(dy)
\\
&&\qquad = \int_{-\infty}^{-x/2} C {\mathrm{e}}^{\gamma(x+y)}
\bar\mu(dy)+ \int_{-x/2}^{+\infty} K\bigl(1+(x +y)^+
\bigr)\bar\mu(dy)
\\
&&\qquad \leq C {\mathrm{e}}^{\gamma(x/2)} + \int_{-x/2}^\infty
K\bigl(1+ |y|\bigr)\bar\mu(dy)
\\
&&\qquad = C {\mathrm{e}} ^{\gamma(x/2)} + \frac
{1}{|x|^{2+\varepsilon}}
\int_{-x/2}^\infty K\bigl(1+ |y|\bigr)|y|^{1+\varepsilon} \bar\mu(dy)
\\
&&\qquad \leq \frac{C}{1+|x|^{1+\varepsilon}},
\end{eqnarray*}
since $\bar\mu$ has a moment of order $2+\varepsilon$. Thus,
$g_\phi
\mathbf{1}_{(- \infty,0]}$ is directly Riemann integrable. Furthermore,
\begin{eqnarray*}
&& \int_{-\infty}^0|x| \bar\mu*f_\phi(x) \,dx
\\
&&\qquad = \int_{-\infty
}^{+\infty} \int
_{-\infty}^0|x| f_\phi(x+y) \,dx \bar\mu(dy)
\\
&&\qquad = \int_{-\infty}^{+\infty} \int_{-\infty}^y|x-y|
f_\phi(x) \,dx \bar\mu(dy)
\\
&&\qquad = \int_{-\infty}^{+\infty} \biggl(\int
_{-\infty}^0|x-y| f_\phi(x) \,dx + \int
_{0}^{y^+}|x-y| f_\phi(x) \,dx \biggr)
\bar\mu(dy)
\\
&&\qquad \leq \int_{-\infty}^{+\infty} \biggl( C \int
_{-\infty}^0|x-y| {\mathrm{e}} ^{\gamma x} \,dx +
2|y| \int_{0}^{y^+} f_\phi(x) \,dx \biggr)
\bar\mu(dy)
\\
&&\qquad \leq C \int_{-\infty}^{+\infty} \bigl(1+|y|+|y|^2
\bigr) \bar\mu(dy)<\infty.
\end{eqnarray*}

\textit{Step} 4. To check that $g_\phi$ is directly Riemann integrable
and $|x g_\phi(x)|$ is integrable for positive $x$, we show that
\[
\sum_{n=0}^{\infty}\sup_{n\leq x <n+1}
\bigl|x g_\phi(x)\bigr| <\infty.
\]
Applying $\mu_0$ invariance of $\nu_0$ and since $A(\psi_0)=A(\psi)$,
we obtain
\[
\bigl|g(x)\bigr|= \biggl|\int\!\!\int\phi\bigl(e^{-x}\bigl(A(\psi) u\bigr) \bigr) - \phi
\bigl(e^{-x}\bigl(\psi(u)\bigr) \bigr)\nu_0(du)
\mu_0(d\psi) \biggr|.
\]
The function $\tilde\phi(x)=\phi({\mathrm{e}}^x)$ is a
Lipschitz on $\R$, hence
\begin{eqnarray*}
&& \bigl|\phi\bigl(e^{-x}\bigl(A(\psi) u\bigr) \bigr) - \phi
\bigl(e^{-x}\psi(u) \bigr) \bigr|
\\
&&\qquad \leq  \min\biggl\{C \biggl|\log
\frac{\psi(u)}{A(\psi)u} \biggr|,2\| \phi\| _\infty\biggr\}
\\
&&\qquad \leq  \min\biggl\{C \biggl|\frac{\psi(u)}{A(\psi) u}-1 \biggr|,2\|\phi\| _\infty
\biggr\}
\\
&&\qquad \leq  \min\biggl\{C \biggl|\frac{B(\psi)}{A(\psi) u} \biggr|,2\|\phi\| _\infty
\biggr\}
=: \rho\biggl(\frac{A u}{B}\biggr),
\end{eqnarray*}
where we use the convention that $\log z=-\infty$ for $z\leq0$ and
$\rho(y):=  \min\{C |\frac{1}{y} |,\break 2\|\phi\|_\infty\}$.
Take now $0\leq n\leq x <n+1$
%
\begin{eqnarray*}
\hspace*{-6pt}&& |x| \bigl|\phi\bigl(e^{-x}(A u) \bigr) - \phi\bigl(e^{-x}
\psi(u) \bigr) \bigr|
\\
\hspace*{-6pt}&&\!\!\qquad \leq\log^+\frac{A u+B}{m}
\\
\hspace*{-6pt}&&\!\!\!\quad\qquad{}\times \rho\biggl(\frac{A u}{B}\biggr) (
\mathbf{1}_{[\log (\psi(u)/(M\mathrm{e}))\leq n\leq
\log(\psi(u)/m)]}+\mathbf{1}_{[\log((A u)/(M\mathrm{e}))\leq n\leq\log((Au)/m)]} ).
\end{eqnarray*}
%
Thus,
\begin{eqnarray*}
\hspace*{-2pt}&& \sum_{n=0}^{\infty}\sup
_{n\leq x <n+1} \bigl|x g_\phi(x)\bigr|
\\
\hspace*{-2pt}&&\qquad \leq  \int\!\!\int\sum
_{n=0}^{\infty} \sup_{n\leq x <n+1}|x| \bigl|\phi
\bigl(e^{-x}(A u) \bigr) - \phi\bigl(e^{-x}\psi(u) \bigr) \bigr|
\nu_0(du)\mu_0(d\psi)
\\
\hspace*{-2pt}&&\qquad \leq \int\!\!\int\log^+\frac{A u+B}{m}\cdot\rho\biggl(\frac{A u}{B}
\biggr) 2\log\frac{{\mathrm{e}}M}{m} \nu_0(du)\mu_0(d\psi)
\\
\hspace*{-2pt}&&\qquad \leq 2\log\frac{{\mathrm{e}}M}{m}
\\
\hspace*{-2pt}&&\quad\qquad{}\times  \int\biggl(\int\biggl(\log^+
\frac{1}m+\log^+ B + \log^+\biggl(\frac{A
u}{B}+1\biggr) \biggr)
\rho\biggl(\frac{A u}{B}\biggr) \nu_0(du) \biggr)
\mu_0(d\psi).
\end{eqnarray*}
To estimate the last expression, we use the fact that
there exists a constant $C$ such
that for all nonincreasing functions $h\dvtx[0,+\infty)\to[0,+\infty)$
and all $M>0$
%
%
\begin{eqnarray}\label{temp}
&& \int_\R h \bigl({|u|}/M \bigr)\nu_0(du)
\nonumber\\[-8pt]\\[-8pt]
&&\qquad \leq C\bigl(1+\log^+M\bigr) \biggl(\|h\| _\infty+\int
_1^{+\infty}h(z) \bigl(1+\log(z)\bigr)
\frac{dz}{z} \biggr).\nonumber
\end{eqnarray}
Before we prove the last inequality, let us check how it implies the lemma.
Since $\log^+(z+1)\rho(z)\leq C/(1+z)^{1/2}$ for $z>0$, by (\ref
{temp}), we have
\begin{eqnarray*}
&& \int\biggl(\log^+\frac{1}m+\log^+ B + \log^+\biggl(
\frac{A u}{B}+1\biggr) \biggr) \rho\biggl(\frac{A u}{B}\biggr)
\nu_0(du)
\\
&&\qquad \leq C \bigl(1+\log^+(B/A)\bigr)
\\
&&\quad\qquad{}\times  \biggl(\bigl(1+\log^+B\bigr)
\\
&&\hspace*{51pt}{} + \int
_{1}^{+\infty} \biggl(\bigl(1+ \log^+ B\bigr)\rho(z) +
\frac{1}{(1+z)^{1/2}} \biggr) \bigl(1+\log^+(z)\bigr) \frac
{dz}{z} \biggr)
\\
&&\qquad \leq C \bigl(1+\bigl(\log^+B\bigr)^2+\log^+ B\log^+A \bigr).
\end{eqnarray*}
The last expression is $\mu_0$-integrable and we conclude.

Finally, to prove (\ref{temp}) we write
\begin{eqnarray*}
&& \int_\R h \bigl({|u|}/M \bigr)\nu_0(du)
\\
&&\qquad \leq \|h\|_\infty\nu_0\bigl([-M{\mathrm{e}},M{\mathrm{e}}]
\bigr)+\int\mathbf{1}_{[|u|>{\mathrm{e}}M]} h \bigl({|u|}/M \bigr)\nu_0(du)
\\
&&\qquad  \leq C\bigl(1+\log^+M\bigr)\|h\|_\infty+ \sum
_{n=1}^\infty\int\mathbf{1}_{[{\mathrm{e}}
^{n+1}M\geq|u|>{\mathrm{e}}^{n} M]}
\nu_0(du)h\bigl(e^n\bigr)
\\
&&\qquad \leq C\bigl(1+\log^+M\bigr)\|h\|_\infty+ \sum
_{n=1}^\infty\bigl(\log^+\bigl({\mathrm{e}}
^{n+1}M\bigr)+ 1 \bigr)h\bigl(e^n\bigr)
\\
&&\qquad \leq C\bigl(1+\log^+M\bigr)\|h\|_\infty+ \sum
_{n=1}^\infty\int_{{\mathrm
{e}}^{n-1}}^{{\mathrm{e}}
^{n}}\bigl(\log^+\bigl(z{\mathrm{e}}^2M\bigr)+ 1 \bigr)h(z)
\frac{dz}{z}
\\
&&\qquad  \leq C\bigl(1+\log^+M\bigr)\|h\|_\infty+ \int_{1}^{\infty}
\bigl(\log^+(z)+\log^+M + 3 \bigr)h(z)\frac{dz}{z}.
\end{eqnarray*}\upqed
\end{pf*}

%
\begin{lem}\label{lemALinterval}
Let $\phi\in C([0,1])$ be a function fixing $0$ and $1$, derivable at
0 and 1 and such that $\phi'(0)=\phi'(1)=:a_{\phi}$. Suppose
\begin{eqnarray*}
\beta^0_1 &=& \inf_{u\in[0,1/2]}
\bigl(1- \phi(u)\bigr)>0,\qquad\beta^0_2 = \inf
_{u \in[0,1/2]}\frac{ \phi(u)}{u}>0,
\\[-1pt]
\beta^0_3 &=& \sup_{u\in[0,1/2]} \biggl\llvert\frac{\phi(u)-a_\phi
u}{u^2}\biggr\rrvert<
\infty,
\\[-1pt]
\beta^1_1 &=& \inf_{u\in[1/2,1]} \phi(u)>0,\qquad
\beta^1_2 = \inf_{u \in[1/2,1]}
\frac{1-\phi(u)}{1-u}>0,
\\[-1pt]
\beta^1_3 &=& \sup_{u\in[1/2,1]} \biggl\llvert\frac{\phi(u)-1-a_\phi
(u-1)}{(u-1)^2}\biggr\rrvert<\infty.
\end{eqnarray*}
Consider the diffeomorphism of $(0,1)$ on $\R$
\[
r(u)=-\frac{1}{u}+\frac{1}{1-u}.
\]
Then $\Psi_\phi=r \circ\phi\circ r^{-1}$ satisfy (\ref{AL}) for $A(\Psi
_\phi
)=1/{a_{\phi}}$ and
\[
B(\Psi_\phi)< C_r \biggl(\frac{(1+a_\phi+\beta_3^0)}{a_{\phi}
\beta^0_2} +
\frac{1}{\beta^0_1} + \frac{(1+a_\phi+\beta_3^1)}{a_{\phi}
\beta^1_2} + \frac{1}{\beta^1_1} \biggr), %
\]
where $C_r$ depends only on the function $r$.
\end{lem}

\begin{pf} Since the function $r$ satisfies $r(u)= -r(1-u)$ and our
assumptions on $\phi$ near 0 and 1 are symmetric, it is sufficient to
prove the condition (\ref{AL}) only for negative $x$.
Since $\beta^0_3<\infty$, by the Taylor expansion we have
%
%
\begin{equation}
\label{diffphi} \phi(u) = au+\epsilon_\phi(u)\qquad\mbox{with }\bigl|
\epsilon_\phi(u)\bigr|\leq\beta_3^0
u^2\mbox{ for $u\leq1/2$}.
\end{equation}
Moreover, simple calculus shows that
%
%
\begin{equation}
\label{diffr} r^{-1}(x) = -\frac{1} x +
\epsilon_{r^{-1}}(x)\qquad\mbox{with } \epsilon_{r^{-1}}(x)= \biggl(
\frac{1}{x^2} \biggr)\mbox{ for $x\to-\infty$}.
\end{equation}
For $x<0$, we write
\[
\biggl| \frac{x}{a_{\phi}} - \Psi_{\phi}(x) \biggr| = \biggl| \frac{x}{a_{\phi}} - r
\bigl( \phi\bigl( r^{-1}(x) \bigr) \bigr) \biggr| \le\biggl| \frac{x}{a_{\phi}} +
\frac{1}{\phi( r^{-1}(x) )} \biggr| + \frac{1}{1- \phi( r^{-1}(x) )}. %
\]
Notice that for $x<0$, $r^{-1}(x)\in(0,1/2)$, therefore, the second
factor can be bounded by $\frac{1}{\beta^0_1}$. So, we need just to
estimate the first term. We write
\[
I(x) = \biggl| \frac{x}{a_{\phi}} - \frac{1}{\phi( r^{-1}(x) )} \biggr| = \frac
{|\phi(r^{-1}(x))x -{a_{\phi}}|}{|{a_{\phi}} \cdot
\phi
(r^{-1}(x))|}.
\]
Take $M = -r(1/10)$, then for $x\in[-M,0]$ we have $\phi
(r^{-1}(x))\geq
\beta^0_2 r^{-1}(x)\geq\beta^0_2/10$ and hence
\[
I(x) \le10 \frac{M+a_{\phi}}{a_{\phi} \beta^0_2}. %
\]
Now we consider $x<-M$. Since there exists $\eta$ such that $x
r^{-1}(x)>\eta$, by (\ref{diffphi}) and (\ref{diffr}), we have
\begin{eqnarray*}
I(x) &=& \frac{ |\phi(r^{-1}(x))x + {a_{\phi}} |\cdot|x| }{
a_{\phi} \cdot (\phi(r^{-1}(x)))/(r^{-1}(x)) \cdot|x r^{-1}(x)|}
\\
&\le&
\frac{1}{a_{\phi} \beta^0_2 \eta} \cdot\bigl|\phi
\bigl(r^{-1}(x)\bigr)x+ {a_{\phi}} \bigr| |x|
\\
&=& \frac{1}{a_{\phi} \beta^0_2 \eta} \bigl| a_\phi r^{-1}(x) x + \varepsilon
_{\phi}\bigl( r^{-1}(x)\bigr) x + a_{\phi} \bigr||x|
\\
&=&
\frac{1}{a_{\phi} \beta^0_2 \eta} \bigl| a_{\phi}\varepsilon_{r^{-1}}(x) x +
\varepsilon_{\phi} \bigl( r^{-1}(x) \bigr) x \bigr|
\\
&\le& \frac{ |a_{\phi} \varepsilon_{r^{-1}}(x) x | +
\beta
_3^0
| ( r^{-1}(x) )^2 x |} {a_{\phi} \beta^0_2 \eta} \leq\frac{C_r(a_\phi
+\beta_3^0)}{a_{\phi} \beta^0_2}.
\end{eqnarray*}\upqed
\end{pf}




%

\printaddresses

\end{document}